\documentclass[12pt]{article}
\usepackage{razb}

\newlength{\longeqmarginwidth}
\newlength{\longeqwidth}
\newlength{\longeqskiplength}
\newcommand{\longeqskip}{\hspace{\longeqskiplength}}
\newcommand{\longeq}[1]{
\settowidth{\longeqmarginwidth}{{\Large \{~\}}~(\theequation)}
\setlength{\longeqwidth}{\textwidth}
\addtolength{\longeqwidth}{-\longeqmarginwidth}
\left. \parbox{\longeqwidth}{\begin{eqnarray*} #1
\end{eqnarray*}} \right\}}

\newcommand{\Vzero}[2]{\mathord{\stackrel{\textstyle\kern1pt\circ}
{\smash V\vbox to6pt{}}}\vphantom{V}_{#1}^{#2}}

\newcommand{\function}[2]{:#1 \longrightarrow #2}
\newcommand{\of}[1]{\left( #1 \right)}

\newcommand{\set}[2]{\left\{\hspace{0.2ex} #1 \left|\: #2
\right. \right\}}

\newcommand{\eval}[2]{\llbracket #1 \rrbracket_{#2}}

\newcommand{\df}{\stackrel{\rm def}{=}}

\newcommand{\im}{{\rm im}}

\newcommand{\rn}{\boldsymbol}

\newcommand{\prob}[1]{ {\bf P}\! \left[ #1 \right] }
\newcommand{\expect}[1]{ {\bf E}\! \left[ #1 \right] }

\newcounter{operator}

\usepackage{amsfonts}
\usepackage{amsmath}
\usepackage{amssymb}
\usepackage{epsfig}
\usepackage{epic,eepic}
\usepackage{stmaryrd}
\newcommand{\injection}[2]{\colon #1 \rightarrowtail #2}
\title{An Extremal Problem Motivated by Triangle-Free Strongly Regular Graphs}
\author{Alexander Razborov\thanks{University of Chicago, USA, {\tt razborov@math.uchicago.edu} and Steklov Mathematical Institute, Moscow, Russia, {\tt razborov@mi.ras.ru}.}}
\begin{document}
\maketitle

\begin{abstract}
We introduce the following combinatorial problem. Let $G$ be a
triangle-free regular graph with edge density\footnote{In this paper all
densities are normalized by $n,\frac{n^2}{2}$ etc. rather than by $n-1,
{n\choose 2},\ldots$} $\rho$. What is the minimum value $a(\rho)$ for which
there always exist two non-adjacent vertices such that the density of their
common neighborhood is $\leq a(\rho)$? We prove a variety of upper bounds
on the function $a(\rho)$ that are tight for the values $\rho=2/5,\ 5/16,\
3/10,\ 11/50$, with $C_5$, Clebsch, Petersen and Higman-Sims being
respective extremal configurations. Our proofs are entirely combinatorial
and are largely based on counting densities in the style of flag algebras.
For small values of $\rho$, our bound attaches a combinatorial meaning to
Krein conditions that might be interesting in its own right. We also prove
that for any $\epsilon>0$ there are only finitely many values of $\rho$
with $a(\rho)\geq\epsilon$ but this finiteness result is somewhat purely
existential (the bound is double exponential in $1/\epsilon$).
\end{abstract}

\section{Introduction} \label{sec:intr}

Triangle-free strongly regular graphs (TFSR graphs), sometimes also called
SRNT (for strongly regular no triangles) is a fascinating object in algebraic
combinatorics. Except for the trivial bipartite series, there are only seven
such graphs known (see e.g. \cite{God}). At the same time, the existing
feasibility conditions still leave out many possibilities. For example, there
are still 66 prospective values of parameters with $\lambda_1\leq 10$, where
$\lambda_1$ is the second largest eigenvalue of $G$ \cite[Tables 1,2]{Big};
the most prominent of them probably being the hypothetical Moore graph of
degree 57. This situation is in sharp contrast with general strongly regular
graphs (or, for that matter. with finite simple groups) where non-trivial
infinite series are abundant, see e.g. \cite[Chapter 10]{GoR}.

Somewhat superficially, the methods employed for studying (triangle-free)
strongly regular graphs can be categorized in ``combinatorial'' and
``arithmetic/algebraic'' methods. The latter are based upon spectral
properties of $G$ or modular counting. The former are to a large extent based
on calculating various quantities (that we will highly prefer to normalize in
such a way that they become densities in $[0,1]$), and these calculations
look remarkably similar to those used in asymptotic extremal combinatorics,
particularly in the proofs based on flag algebras. The unspoken purpose of
this paper is to highlight and distill these connections between the two
areas. To that end, we introduce and study a natural extremal problem
corresponding to strong regularity.

Before going into some technical details, it might be helpful to digress on
the apparent contradiction of studying highly symmetric and inherently finite
objects with methods that are quite analytical and continuous in their
nature. The key to resolving this is the simple observation that has been
used in extremal combinatorics many times: any finite graph (or, for that
matter, more complicated combinatorial object) can be alternately viewed as
an analytical object called its {\em stepfunction graphon} \cite[\S
7.1]{Lov4} or, in other words, {\em infinite blow-up}. It is obtained by
replacing every vertex with a measurable set of appropriate measure. To this
object we can already apply all methods based on density calculations, and
the conversion of the results back to the finite world is straightforward.

\bigskip
Let us now fix some notation. All graphs $G$ in this paper are simple and, unless otherwise noted, triangle-free. By $n=n(G)$  we always denote the number of vertices, and let
$$
\rho=\rho(G) \df \frac{2|E(G)|}{n(G)^2}
$$
be the edge density of $G$. Note that the normalizing factor here is
$\frac{n^2}2$, not ${n\choose 2}$: the previous paragraph provides a good
clue as why this is much more natural choice. A {\em $\rho$-regular graph} is
a regular graph $G$ with $\rho(G)=\rho$.  We let
$$
a(G)\df \min_{(u,v)\not\in E(G)} \frac{|N_G(u)\cap N_G(v)|}{n(G)},
$$
where $N_G(v)$ is the vertex neighbourhood of $v$. For a rational number $\rho\in [0,1/2]$, we let
\begin{equation} \label{eq:a_rho}
a(\rho) \df \max\set{a(G)}{G\ \text{is a triangle-free}\ \rho\text{-regular graph}}
\end{equation}
Our goal is to give upper bounds on $a(\rho)$.

\begin{remark} \label{rem:maxvssup}
We stress that we do have here maximum, not just supremum, this will be
proven below (see Corollary \ref{cor:finiteness}). In particular, $a(\rho)$
is also rational. Another finiteness result (Corollary \ref{cor:finiteness2})
says that for every $\epsilon>0$ there exist only finitely many rationals
$\rho$ with $a(\rho)\geq\epsilon$. While this result is of somewhat
existential nature (the bound is double exponential in $1/\epsilon$), it
demonstrates, somewhat surprisingly, that our relaxed version of strong
regularity still implies at least some rigidity properties that might be
expected from much more symmetric structures in algebraic combinatorics.
\end{remark}

\begin{remark} \label{rem:no_graphons} The definition of $a(G)$ readily extends to graphons,
and it is natural to ask whether this would allow us to extend the definition
of $a(\rho)$ to irrational $\rho$ or at least come up with interesting
constructions beyond finite graphs: such constructions are definitely not
unheard of in the extremal combinatorics. Somewhat surprisingly (again), the
answer to both questions is negative. Namely, we have the dichotomy: every
triangle-free graphon $W$ (we do not even need regularity here) is either a
finite stepfunction of a finite vertex-weighted graph or satisfies $a(W)=0$
(Theorem \ref{thm:no_graphon}).
\end{remark}

\begin{remark} Every TFSR graph $G$ with parameters $(n,k,c)$, where $k$ is the
degree and $c$ is the size of common neighbourhoods of non-adjacent vertices
leads to the lower bound $a(k/n)\geq c/n$. Thus, optimistically, one could view upper bounding the function $a(\rho)$ as an approach to finding more feasibility conditions for TFSR graphs based on entirely combinatorial methods. This hope is somewhat supported by the fact that our bound is tight for the values corresponding to four (out of seven) known TSFR graphs, as well as an infinite sequence of values not ruled out by other conditions.
\end{remark}

\begin{remark} \label{rem:large_density}
As we will see below, in the definition \eqref{eq:a_rho} we can replace
ordinary $\rho$-regular triangle-free graphs with weighted twin-free
$\rho$-regular triangle-free graphs that can be additionally assumed to be
maximal. A complete description of such graphs with $\rho>1/3$ was obtained
in \cite{BrT}. Along with very simple Lemma \ref{lem:linear} below, this
allows us to completely compute the value of $a(\rho)$ for $\rho>1/3$ and, in
particular, determine those values of $\rho$ for which $a(\rho)> 0$. Using
relatively simple methods from Section \ref{sec:combinatorial}, we can prove
the bounds $a(\rho)\leq \frac{\rho}3\ (1/3\leq\rho\leq 3/8)$, $a(\rho)\leq
3\rho-1\ (3/8\leq \rho\leq 2/5)$ and $a(\rho)=0\ (2/5<\rho<1/2)$. But since
they are significantly inferior (that is, for $\rho<2/5$) to those that
follow from \cite{BrT}, we will save space and {\bf in the rest of the paper
focus on the range $\rho\leq 1/3$}.
\end{remark}

Our main result is shown on Figure \ref{fig:main}.
\begin{figure}[tb]
\begin{center}
\epsfig{file=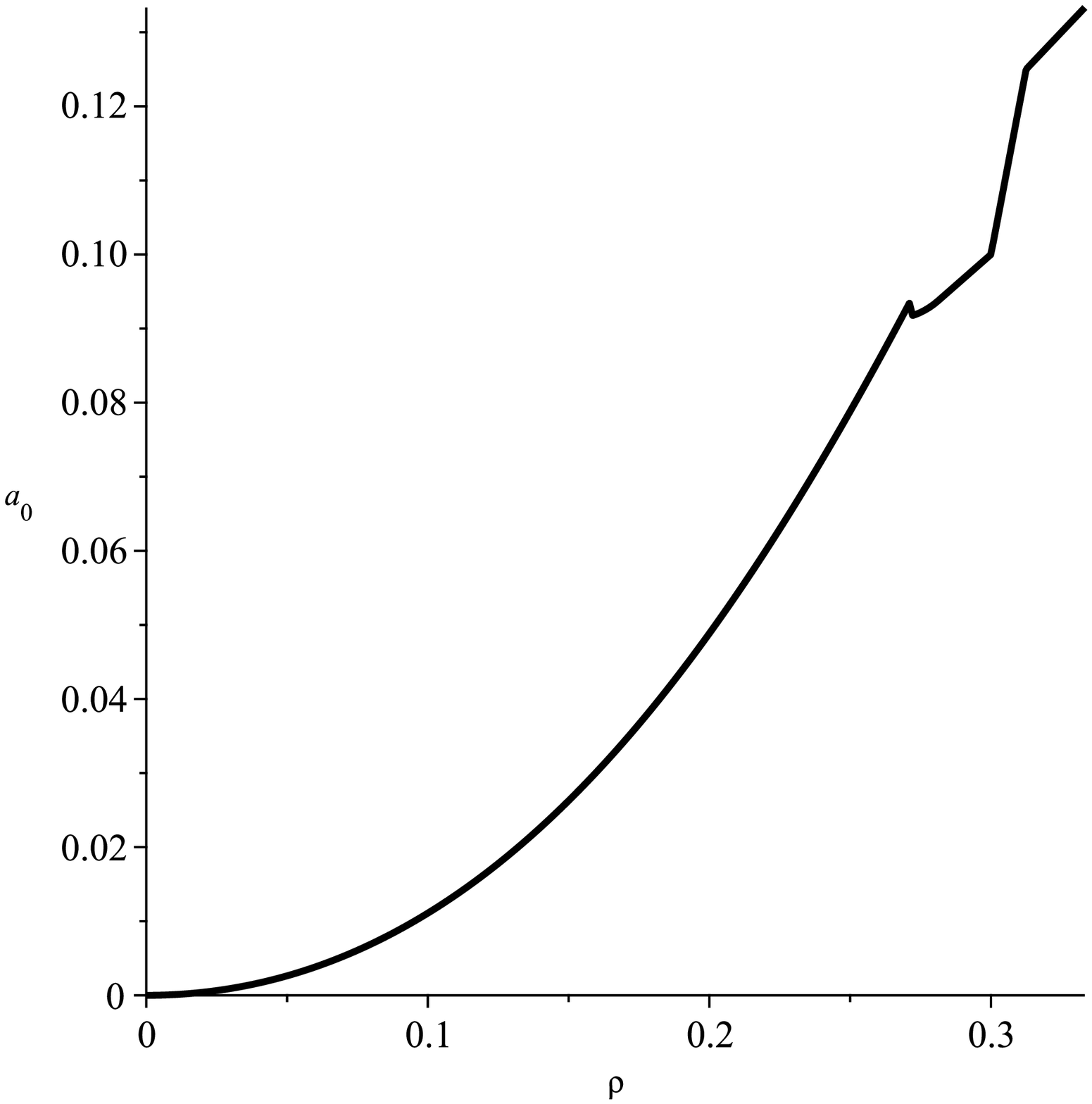,width=10cm}\\ \vspace{1cm} \epsfig{file=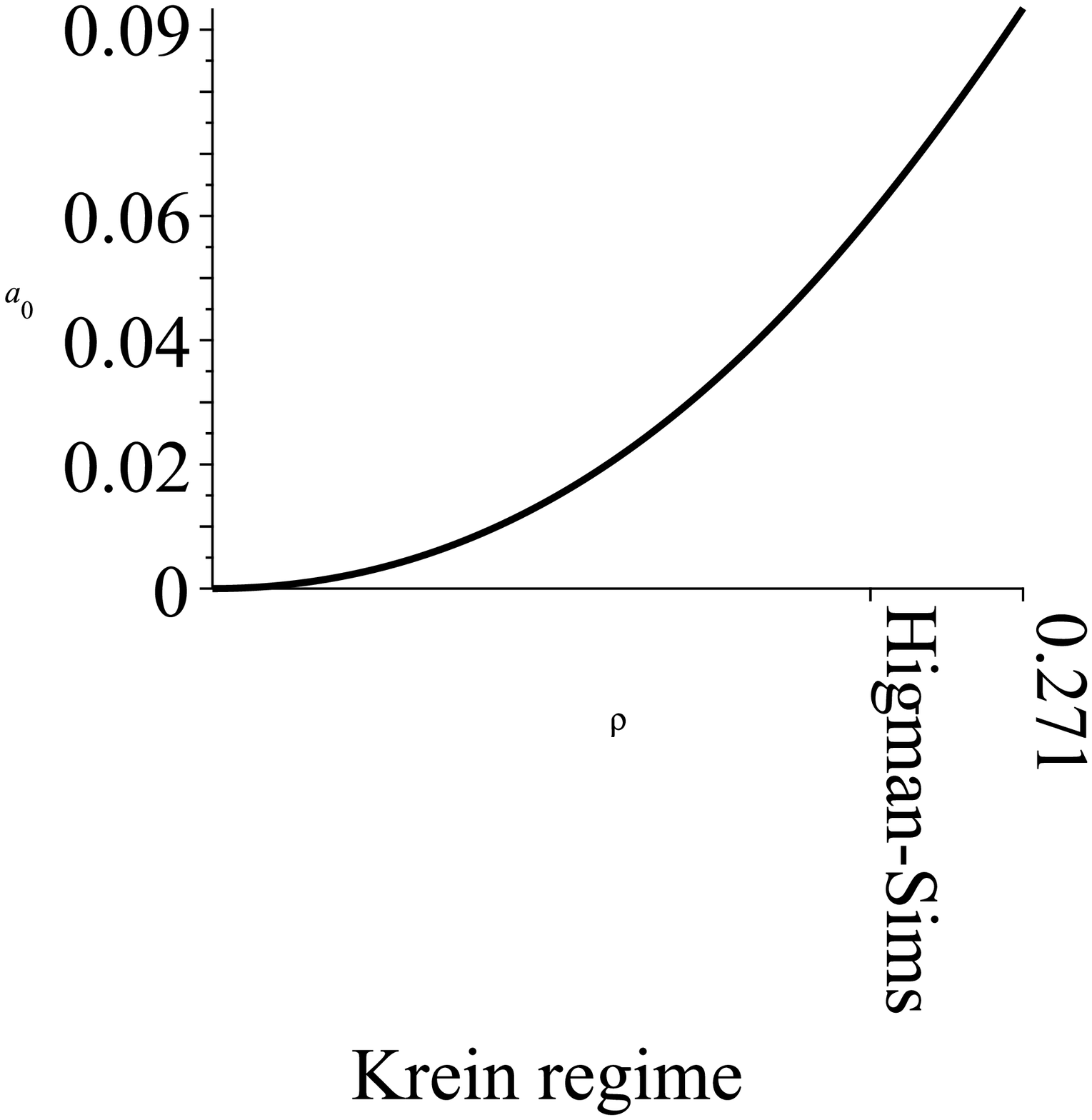,width=5cm} \hspace{3cm} \epsfig{file=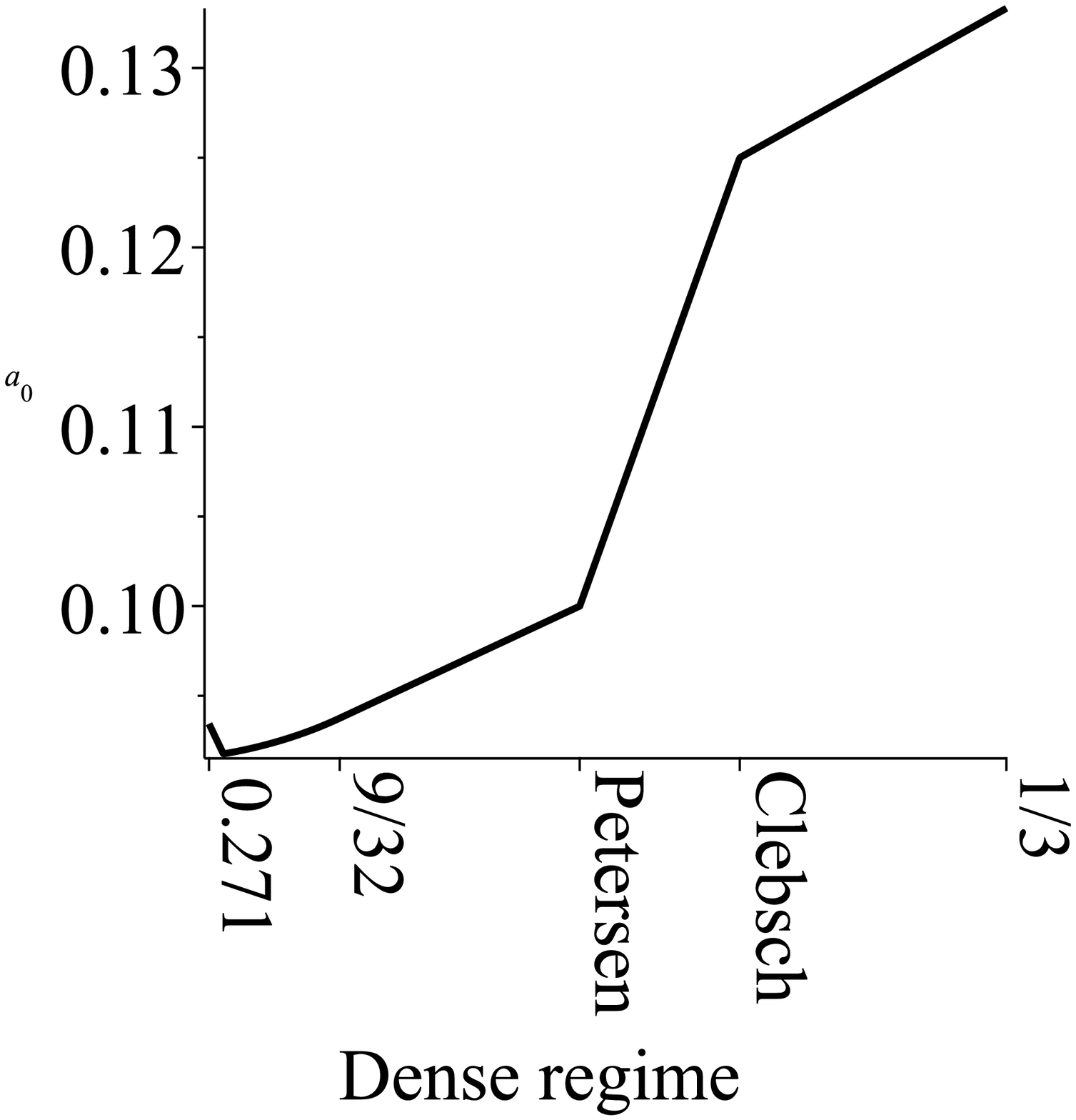,width=5cm}
\vspace{1cm}
\caption{The main result \label{fig:main}}
\end{center}
\end{figure}
The analytical expressions for our upper bound $a_0(\rho)$ will be given in
Theorem \ref{thm:main}; for now let us briefly comment on a few features of
Figure \ref{fig:main}.

\begin{remark}
 The bound is tight for the values $\rho=\frac{11}{50}, \frac 3{10},
    \frac 5{16}$ corresponding to Higman-Sims, Petersen and Clebsch,
    respectively. It is piecewise linear for $\rho\geq 9/32$ and involves three
    algebraic functions of degree $\leq 4$ when $\rho\leq 9/32$.
 \end{remark}

\clearpage

\begin{remark} \label{rem:krein}
 Let us explain the reasons for using the term ``Krein bound''. It may not be seen well
 on Figure \ref{fig:main} but this curve has a singular point at
\begin{equation}\label{eq:rho0}
  \rho_0\df \frac 3{98}(10-\sqrt{2})\approx 0.263.
\end{equation}
For $\rho\geq \rho_0$, $a_0(\rho)$ is a solution to a polynomial equation
$g_K(\rho,a) =0$ that is most likely an artifact of the proof method (and it
gets superseded at $\rho\approx 0.271$ by other methods anyway). The bound
for $\rho\leq \rho_0$ is more interesting.

 Recall
     (see e.g. \cite[Chapter 10.7]{GoR}) that the Krein parameters $K_1,K_2$
     provide powerful constraints $K_1\geq 0,\ K_2\geq 0$ on the existence
     of strongly regular graphs, and in the special
     case of triangle-free graphs we are interested in this paper they can
     be significantly simplified \cite{Big}.

     Now, $K_1,K_2$ are rational functions of $k,c$ {\em and} non-trivial
     eigenvalues $\lambda_1,\lambda_2$ of the adjacency matrix. As
     such, when written as functions of $k,c$, they become (conjugate) algebraic
     quadratic functions and thus do not seem to possess any obvious
     combinatorial meaning. Their {\em product}, however, is the {\em
     rational} function in $k,c$:
     \begin{equation} \label{eq:K1K2}
     K_1K_2=(k-1)(k-c)(k^2-k(3c+1)-c^3+4c^2-c)\geq 0
     \end{equation}
     Re-writing the non-trivial term here in the variables $\rho=k/n,\
     c=a/n$ (and recalling that $n=1+ \frac{k(k-1+c)}c$), we will get a
     constraint $f_K(\rho,a)\geq 0$ that holds for all TFSR graphs. What we
     prove with purely combinatorial methods is that for $\rho\leq\rho_0$ (and
     less us remark that all hypothetical TFSR graphs are confined to that
     region) this inequality holds in much less rigid setting.

\smallskip
     As a by-side heuristical remark, this bound was discovered by flag-algebraic
     computer experiments with particular values of $\rho$ corresponding to
     potential TFSR graphs from \cite[Tables 1,2]{Big}. The result turned out
     to be tight precisely for those values for which
     $c=\lambda_1(\lambda_1-1)$, which is equivalent to $K_2=0$. The connection
     to Krein parameters and, as a consequence, the hypothesis $f_K(\rho,a)\geq 0$
     suggested itself immediately.
\end{remark}

\section{Preliminaries} \label{sec:prel}

We utilize all notation introduced in the previous section. In particular,
all graphs $G=(V(G),E(G))$ are simple and, unless otherwise noted,
triangle-free, and $n=n(G)$ is the number of vertices.

Let us now remind some rudimentary notions from the language of flag algebras
(see \cite[\S 2.1]{flag}) restricted to graphs. A {\em type} $\sigma$ is
simply a totally labelled graph, that is a graph on the vertex set
$[k]\df\{1,2,\ldots,k\}$ for some $k$ called the {\em size} of $\sigma$.
Figure \ref{fig:types} shows all types used in this paper, including the
trivial type 0 of size 0.
\begin{figure}[ht]
\setlength{\unitlength}{0.254mm}
\begin{picture}(557,237)(30,-261)
        \special{color rgb 0 0 0}\put(110,-116){\shortstack{$1$}} 
        \special{color rgb 0 0 0}\allinethickness{0.254mm}\special{sh 0.99}\put(115,-90){\ellipse{4}{4}} 
        \special{color rgb 0 0 0}\allinethickness{0.254mm}\special{sh 0.99}\put(190,-90){\ellipse{4}{4}} 
        \special{color rgb 0 0 0}\allinethickness{0.254mm}\special{sh 0.99}\put(240,-90){\ellipse{4}{4}} 
        \special{color rgb 0 0 0}\put(30,-116){\shortstack{0}} 
        \special{color rgb 0 0 0}\put(105,-86){\shortstack{\scriptsize 1}} 
        \special{color rgb 0 0 0}\put(210,-116){\shortstack{$N$}} 
        \special{color rgb 0 0 0}\put(230,-86){\shortstack{\scriptsize 2}} 
        \special{color rgb 0 0 0}\put(180,-86){\shortstack{\scriptsize 1}} 
        \special{color rgb 0 0 0}\allinethickness{0.254mm}\special{sh 0.99}\put(395,-90){\ellipse{4}{4}} 
        \special{color rgb 0 0 0}\allinethickness{0.254mm}\special{sh 0.99}\put(445,-90){\ellipse{4}{4}} 
        \special{color rgb 0 0 0}\put(415,-116){\shortstack{$\mathcal I$}} 
        \special{color rgb 0 0 0}\put(435,-86){\shortstack{\scriptsize 2}} 
        \special{color rgb 0 0 0}\put(385,-86){\shortstack{\scriptsize 1}} 
        \special{color rgb 0 0 0}\allinethickness{0.254mm}\special{sh 0.99}\put(420,-40){\ellipse{4}{4}} 
        \special{color rgb 0 0 0}\put(410,-36){\shortstack{\scriptsize 3}} 
        \special{color rgb 0 0 0}\allinethickness{0.254mm}\special{sh 0.99}\put(495,-90){\ellipse{4}{4}} 
        \special{color rgb 0 0 0}\allinethickness{0.254mm}\special{sh 0.99}\put(545,-90){\ellipse{4}{4}} 
        \special{color rgb 0 0 0}\put(515,-116){\shortstack{$\mathcal P$}} 
        \special{color rgb 0 0 0}\put(530,-86){\shortstack{\scriptsize 2}} 
        \special{color rgb 0 0 0}\put(485,-86){\shortstack{\scriptsize 1}} 
        \special{color rgb 0 0 0}\allinethickness{0.254mm}\special{sh 0.99}\put(520,-40){\ellipse{4}{4}} 
        \special{color rgb 0 0 0}\put(510,-36){\shortstack{\scriptsize 3}} 
        \special{color rgb 0 0 0}\allinethickness{0.254mm}\path(495,-90)(520,-40) 
        \special{color rgb 0 0 0}\allinethickness{0.254mm}\path(520,-40)(545,-90) 
        \special{color rgb 0 0 0}\allinethickness{0.254mm}\special{sh 0.99}\put(295,-90){\ellipse{4}{4}} 
        \special{color rgb 0 0 0}\allinethickness{0.254mm}\special{sh 0.99}\put(345,-90){\ellipse{4}{4}} 
        \special{color rgb 0 0 0}\put(315,-116){\shortstack{$E$}} 
        \special{color rgb 0 0 0}\put(335,-86){\shortstack{\scriptsize 2}} 
        \special{color rgb 0 0 0}\put(285,-86){\shortstack{\scriptsize 1}} 
        \special{color rgb 0 0 0}\allinethickness{0.254mm}\path(295,-90)(345,-90) 
        \special{color rgb 0 0 0}\allinethickness{0.254mm}\special{sh 0.99}\put(145,-235){\ellipse{4}{4}} 
        \special{color rgb 0 0 0}\allinethickness{0.254mm}\special{sh 0.99}\put(195,-235){\ellipse{4}{4}} 
        \special{color rgb 0 0 0}\put(165,-261){\shortstack{$\mathcal Q_1$}} 
        \special{color rgb 0 0 0}\put(180,-231){\shortstack{\scriptsize 2}} 
        \special{color rgb 0 0 0}\allinethickness{0.254mm}\special{sh 0.99}\put(170,-185){\ellipse{4}{4}} 
        \special{color rgb 0 0 0}\put(160,-181){\shortstack{\scriptsize 3}} 
        \special{color rgb 0 0 0}\allinethickness{0.254mm}\path(145,-235)(170,-185) 
        \special{color rgb 0 0 0}\put(135,-231){\shortstack{\scriptsize 1}} 
        \special{color rgb 0 0 0}\allinethickness{0.254mm}\special{sh 0.99}\put(250,-235){\ellipse{4}{4}} 
        \special{color rgb 0 0 0}\allinethickness{0.254mm}\special{sh 0.99}\put(300,-235){\ellipse{4}{4}} 
        \special{color rgb 0 0 0}\put(270,-261){\shortstack{$\mathcal Q_2$}} 
        \special{color rgb 0 0 0}\put(285,-231){\shortstack{\scriptsize 2}} 
        \special{color rgb 0 0 0}\allinethickness{0.254mm}\special{sh 0.99}\put(275,-185){\ellipse{4}{4}} 
        \special{color rgb 0 0 0}\put(265,-181){\shortstack{\scriptsize 3}} 
        \special{color rgb 0 0 0}\put(240,-231){\shortstack{\scriptsize 1}} 
        \special{color rgb 0 0 0}\allinethickness{0.254mm}\path(275,-185)(300,-235) 
        \special{color rgb 0 0 0}\allinethickness{0.254mm}\special{sh 0.99}\put(355,-235){\ellipse{4}{4}} 
        \special{color rgb 0 0 0}\allinethickness{0.254mm}\special{sh 0.99}\put(405,-235){\ellipse{4}{4}} 
        \special{color rgb 0 0 0}\put(375,-261){\shortstack{$\mathcal D$}} 
        \special{color rgb 0 0 0}\put(390,-231){\shortstack{\scriptsize 2}} 
        \special{color rgb 0 0 0}\put(345,-231){\shortstack{\scriptsize 1}} 
        \special{color rgb 0 0 0}\allinethickness{0.254mm}\special{sh 0.99}\put(380,-185){\ellipse{4}{4}} 
        \special{color rgb 0 0 0}\put(370,-181){\shortstack{\scriptsize 3}} 
        \special{color rgb 0 0 0}\allinethickness{0.254mm}\path(355,-235)(380,-185) 
        \special{color rgb 0 0 0}\allinethickness{0.254mm}\path(380,-185)(405,-235) 
        \special{color rgb 0 0 0}\allinethickness{0.254mm}\special{sh 0.99}\put(430,-185){\ellipse{4}{4}} 
        \special{color rgb 0 0 0}\put(420,-181){\shortstack{\scriptsize 4}} 
        \special{color rgb 0 0 0}\allinethickness{0.254mm}\path(355,-235)(430,-185) 
        \special{color rgb 0 0 0} 
\end{picture}
\caption{\label{fig:types} Types}
\end{figure}

A {\em flag} is a graph partially labelled by labels from $[k]$ for some
$k\geq 0$. Every flag $F$ belongs to the unique type obtained by removing all
unlabelled vertices. Figure \ref{fig:flags} lists all flags we need in this
paper.
\begin{figure}[ht]
\input{flags.eepic}
\caption{\label{fig:flags} Flags}
\end{figure}

Mnemonic rules used in this notation are reasonably consistent: the
subscript, when present, normally denotes the overall number of vertices in
the flag. The first part of the superscript denotes the type of the flag. The
remaining part, when present, helps to identify the flag in case of
ambiguity. For example, there is only one flag $P_3^N$ based on the path of
length 2 and the type $N$. There are, however, two flags based on its
complement $\bar P_3$, and $\bar P_3^{N.c}$ [$\bar P_3^{N.b}$] is the flag in
which the first labelled vertex is the central [border, respectively] vertex
in $\bar P_3$.

\clearpage

Also, for $S\subseteq [3]$ we denote by $F^{\mathcal I}_S$ the flag with 3
labelled independent vertices and one unlabelled vertex connected to the
vertices from $S$. Thus, $S_4^{\mathcal I} = F^{\mathcal I}_{\{1,2,3\}}$ and
$T_4^{\mathcal I} = F^{\mathcal I}_{\{3\}}$.

\medskip
Let $F$ be a flag of type $\sigma$ with $k$ labelled vertices and $\ell-k$
unlabelled ones, and $v_1,\ldots,v_k$ be ({\em not} necessarily distinct)
vertices in the target graph $G$ that span the type $\sigma$, that is
$(v_i,v_j)\in E(G)$ if and only if $(i,j)\in E(\sigma)$. Then we let
$F(v_1,\ldots,v_k)$ be the probability that after picking
$\rn{w_{k+1}},\ldots, \rn{w_\ell}\in V(G)$ independently at random, the
$\sigma$-flag induced in $G$ by $v_1,\ldots,v_k,\rn{w_{k+1}},\ldots,
\rn{w_\ell}$ is isomorphic (in the label-preserving way) to $F$. We stress
that $\rn{w_{k+1}},\ldots,\rn{w_\ell}$ are chosen completely independently at
random; in particular some or all of them may be among $\{v_1,\ldots,v_k\}$.
When this happens, we treat colliding vertices as non-adjacent twins.

We will also need some basic operations on flags (multiplication, evaluation
and lifting operators, to be exact) but since they will not be needed until
Section \ref{sec:analytical}, we defer it until then.

In this notation $\rho=\frac{2|E(G)|}{n^2}$ is the edge density,
$e(v)=\frac{|N_G(v)|}{n}$ is the relative degree of $v$ and $P_3^N(u,v) =
\frac{|N_G(u)\cap N_G(v)|}{n^2}$ is the relative size of the common
neighbourhood of $u$ and $v$. A graph $G$ is {\em $\rho$-regular} if
$e(v)\equiv \rho$. Etc.

\smallskip
{\bf Warning.} When evaluating [the density of] say $C_4$, we must take into
account not only induced copies, but also contributions made by paths $P_3$
(one collapsing diagonal) and even by edges (both diagonals collapsing).

\smallskip
We let
$$
a(G) \df \min_{(u,v)\not\in E(G)} P_3^N(u,v)
$$
and, for a rational $\rho\in [0,1/2]$, we also let
$$
a(\rho)\df\max \set{a(G)}{G\ \text{a triangle-free $\rho$-regular graph}}
$$
(we will prove below that the minimum value here is actually attained).

\section{The statement of the main result}

Many of our statements and proofs, particularly for small values of $\rho$, involve rather cumbersome computations.
A Maple worksheet with supporting evidence can be found at {\tt http://people.cs.uchicago.edu/\~{}razborov/files/tfsr.mw}

Let\footnote{This is the non-trivial factor in \eqref{eq:K1K2} re-written in terms of $\rho,a$}
$$
f_K(\rho,a) \df a^3+(3\rho-4)a^2 + (5\rho-1)a -4\rho^3+\rho^2.
$$
Then
$$
f_K(\rho,\rho^2) =\rho^3(\rho^3+3\rho^2-4\rho+1)>0
$$
(since $\rho\leq 1/3$) while

$$
f_K\of{\rho,\frac{\rho^2}{1-\rho}} = -\frac{\rho^5(1-2\rho)}{(1-\rho)^3}<0.
$$
Let $\mathsf{Krein}(\rho)$ be the largest (actually, the only) root of the
cubic polynomial equation $f_K(\rho,z)=0$ in the interval\footnote{The left
end of this interval is determined entirely by convenience, but the right end
represents a trivial upper bound on $a(\rho)$ resulting from double counting
copies of $C_4$. See the calculation after \eqref{eq:trivial_bound} for more
details.} $z\in \left[ \rho^2, \frac{\rho^2}{1-\rho}\right]$.

Next, let
\begin{eqnarray*}
&& g_K(\rho,a) \df a^4+a^3((4\sqrt 2-8)\rho +7-4\sqrt 2) + a^2\rho((6-4\sqrt 2)\rho +8\sqrt 2 -13)\\
&& \longeqskip +a\rho(\rho^2 +(15-10\sqrt 2)\rho+2\sqrt 2-3)+\rho^3((8\sqrt 2-12)\rho+3-2\sqrt 2)
\end{eqnarray*}
(the meaning of this expression might become clearer in Section \ref{sec:krein}). We again have $g_K(\rho,\rho^2)>0$,
\begin{equation} \label{eq:g_K}
g_K\of{\rho,\frac{\rho^2}{1-\rho}} =-\frac{\rho^7(1-2\rho)}{(1-\rho)^4}<0,
\end{equation}
and we define $\widehat{\mathsf{Krein}}(\rho)$ as the largest (unique) root of the equation $g_K(\rho,z)=0$ in the interval $z\in \left[\rho^2, \frac{\rho^2}{1-\rho}\right]$.

We note that $\mathsf{Krein}(\rho_0)= \widehat{\mathsf{Krein}}(\rho_0) =
\frac{\rho_0}{3}$ (recall that $\rho_0$ is given by \eqref{eq:rho0}), and
that they have the same first derivative at $\rho=\rho_0$ as well. It should
also be noted that $\widehat{\mathsf{Krein}}(\rho)\geq \mathsf{Krein}(\rho)$
and that they are very close to each other. For example, let
$$
\rho_1\approx 0.271
$$
be the appropriate root of the equation $g_K(\rho,\frac{1-3\rho}2)=0$; this
is the point at which Krein bounds yield to more combinatorial methods, see
Figure \ref{fig:main}. Then in the relevant interval $\rho\in[\rho_0,\rho_1]$
we have $\widehat{\mathsf{Krein}}(\rho) \leq \mathsf{Krein}(\rho)+3\cdot
10^{-6}$.

We finally let
$$
\mathsf{Improved}(\rho) \df \frac{15-22\rho-2\sqrt{242\rho-27-508\rho^2}}{74},
$$
and let
$$
\rho_2 \df \frac{66+2\sqrt{13}}{269} \approx 0.272
$$
be the root of the equation $\mathsf{Improved(\rho)}=\frac {1-3\rho}{2}$.

We can now explain Figure \ref{fig:main} as follows:

\begin{theorem} \label{thm:main}
For $\rho\leq 1/3$ we have $a(\rho)\leq a_0(\rho)$, where
$$
a_0(\rho)\df
\begin{cases}
\mathsf{Krein}(\rho), & \rho\in [0,\rho_0]\\
\widehat{\mathsf{Krein}}(\rho), & \rho\in [\rho_0,\rho_1]\\
\frac{1-3\rho}2, & \rho \in [\rho_1, \rho_2]\\
\mathsf{Improved}(\rho), & \rho\in [\rho_2,9/32]\\
\rho/3, & \rho\in [9/32,3/10]\\
2\rho-\frac 12, & \rho\in [3/10, 5/16]\\
\frac 25\rho, & \rho\in [5/16,1/3].
\end{cases}
$$
\end{theorem}

\section{Finiteness results} \label{sec:boundedeness}

Before embarking on the proof of Theorem \ref{thm:main}, let us fulfill the promise made in Remarks \ref{rem:maxvssup} and \ref{rem:no_graphons}.

Throughout the paper we will be mostly working with (vertex)-weighted graphs,
i.e. with graphs $G$ equipped with a probability measure $\mu$ on $V(G)$,
ordinary graphs corresponding to the uniform measure. The flag-algebraic
notation $F(v_1,\ldots,v_k)$ introduced in Section \ref{sec:prel} readily
extends to this case simply by changing the sampling distribution from
uniform to $\mu$.

The {\em twin relation $\approx$} on $G$ is given by $u\approx v$ iff $N_G(u)=N_G(v)$, and a graph $G$ is {\em twin-free} if its twin relation is trivial. Factoring a graph by its twin relation gives us a {\bf twin-free weighted} graph $G^{\text{red}}$ that preserves all properties of the original graph $G$ (like the values $\rho(G)$ and $a(G)$, $\rho$-regularity or triangle-freeness) we are interested in this paper.

Our main technical argument in this section is the following

\begin{theorem} \label{thm:boundedeness}
Let $(G,\mu)$ be a vertex-weighted triangle-free twin-free graph and $a\df a(G,\mu)$. Then
$$
n(G) \leq (2a^{-1})^{1+a^{-1}} +2a^{-1}.
$$
\end{theorem}
\begin{proof}
Let $n\df n(G)$ and $V(G) \df \{v_1,\ldots,v_n\}$, where $\mu(v_1)\geq\ldots
\geq \mu(v_n)$. Choose the maximal $k$ with the property
$\mu(\{v_k,\ldots,v_n\})\geq a/2$.  Then, by averaging, we have
$\frac{1-a/2}{k-1}\geq\frac{a/2}{n-k+1}$ which is equivalent to
$$
n\leq 2a^{-1}(n-k+1).
$$
Hence, denoting
$$
W_0\df \{v_{k+1},\ldots,v_n\}
$$
(note for the record that $\mu(W_0)< a/2$), it suffices to prove that
\begin{equation} \label{eq:w0bound}
|W_0| \leq (2a^{-1})^{a^{-1}}.
\end{equation}

For $W\subseteq V(G)$ let us define
$$
K(W) \df \bigcap_{w\in W} N_G(w);
$$
note that $K(W)\cap W=\emptyset$.
The bound \eqref{eq:w0bound} will almost immediately follow from the following
two claims.

\begin{claim} \label{clm:easy}
For any $W\subseteq V(G)$ and $v^\ast\not\in W\cup K(W)$ we have
$$
\mu\of{\of{\bigcup_{v\in K(w)}N_G(v)}\cup N_G(v^\ast)} \geq
\mu\of{\bigcup_{v\in K(w)}N_G(v)} +a.
$$
\end{claim}
\begin{proofof}{Claim \ref{clm:easy}}
Since $v^\ast\not\in K(W)$, there exists $w\in W$ such that
$(v^\ast,w)\not\in E(G)$; moreover, $w\neq v^\ast$ since $v^\ast\not\in W$.
Now, all vertices in $N_G(v^\ast)\cap N_G(w)$ contribute to the difference
$N_G(v^\ast)\setminus \bigcup_{v\in K(w)}N_G(v)$ (since $w\in W$ and $G$ is
triangle-free).
\end{proofof}

\begin{claim} \label{clm:difficult}
For every $W\subseteq V(G)$ with $\mu(W)\leq a/2$ and $|W|\geq 2$ there
exists $v^\ast\not\in W\cup K(W)$ such that\footnote{note that this bound is
about absolute {\bf sizes}, not about measures}
$$
|W\cap N_G(v^\ast)| \geq \frac a2|W|.
$$
\end{claim}
\begin{proofof}{Claim \ref{clm:difficult}}
Let
$$
L(W) \df \set{v\not\in W}{N_G(v)\cap W\not\in \{\emptyset,W\}}.
$$
Note that $L(W)$ is disjoint from both $W$ and $K(W)$ and that there are no
edges between $K(W)$ and $L(W)$. The desired vertex $v^\ast$ will belong to
$L(W)$, and we consider two (similar) cases.

{\bf Case 1.} $K(W)=\emptyset$.

\noindent In this case we have
\begin{equation} \label{eq:Lw}
L(W) = \of{\bigcup_{w\in W} N_G(w)}\setminus W.
\end{equation}
W.l.o.g. we can assume that $n\geq 3$ which implies (since $G$ is twin-free)
that $G$ is not a star. That is, for every $w\in V(G)$ there exists $v\neq w$
non-adjacent to it and hence we have the bound $e(w)\geq P_3^N(v,w)\geq a$ on
the minimum degree. Along with \eqref{eq:Lw} and the assumption $\mu(w)\leq
a/2$, we get $\mu(N_G(w)\cap L(W))\geq a/2$ for any $w\in W$. Now the
existence of the required $v^\ast\in L(W)$ follows by standard double
counting of edges between $W$ and $L(W)$ (note that, unlike $L(W)$, the set
$W$ is {\bf not} weighted in this argument according to $\mu$).

\smallskip
{\bf Case 2.} $K(W)\neq\emptyset$.

\noindent Then $W$ is independent and the condition $v\not\in W$ in the
definition of $L(W)$ can be dropped. Fix arbitrarily $w\neq w'\in W$ (this is
how we use the assumption $|W|\geq 2$). Then $w,w'$ are not twins and
$N_G(w)\triangle N_G(w')\subseteq L(W)$, hence $L(W)\neq\emptyset$. Fix
arbitrarily $v\in L(W)$ and $w\in W$ with $(v,w)\not\in E(G)$. Then
\begin{equation} \label{eq:p3vsl}
N_G(v) \cap N_G(w) \subseteq L(W)
\end{equation}
(since there are no edges between $L(W)$ and $K(W)$) hence $\mu(L(W))\geq a$.
We claim that actually $\mu(N_G(w)\cap L(W))\geq a$ for {\bf every} $w\in W$.
Indeed, if $N_G(w)\supseteq L(W)$ this follows from the bound we have just
proved, and if there exists $v\in L(W)$ with $(v,w)\not\in E(G)$, this
follows from \eqref{eq:p3vsl}. The analysis of Case 2 is now completed by the
same averaging argument as in Case 1 (with the final bound improved by a
factor of two).
\end{proofof}

The rest of the proof of Theorem \ref{thm:boundedeness} is easy. We start
with the set $W_0$ and then, using Claims \ref{clm:difficult} and
\ref{clm:easy}, recursively construct sets $W_0\supset W_1\supset
W_2\supset\ldots$ such that\footnote{We could have shaved off an extra factor
$2^{r-1}$ by observing that Case 1 in Claim \ref{clm:difficult} may occur at
most once.} $|W_r|\geq (2a^{-1})^r|W_0|$ and
\begin{equation}\label{eq:termination}
   \mu\of{\bigcup_{v\in K(W_r)}N_G(v)}\geq ar.
\end{equation}
This process may terminate for only one reason: when the assumption
$|W_r|\geq 2$ from Claim \ref{clm:difficult} no longer holds. On the
other hand, due to \eqref{eq:termination}, it must terminate within $a^{-1}$
steps. The bound \eqref{eq:w0bound} follows, and this also completes the
proof of Theorem \ref{thm:boundedeness}.
\end{proof}

\begin{remark}
The bound in Theorem \ref{thm:boundedeness} is essentially tight. Indeed, let
us consider the graph $G_h$ on $n=2h+2^h$ vertices
$$
\set{u_{i\epsilon}}{i\in [h],\ \epsilon\in \{0,1\}} \stackrel .\cup \set{v_a}{a\in \{0,1\}^h},
$$
and let $E(G_h)$ consist of the matching $\set{(u_{i0}, u_{i1})}{i\in
[h]}\cup \set{(v_a,v_{1-a})}{a\in\{0,1\}^h}$ as well as the cross-edges
$\set{(u_{i\epsilon}, v_a)}{a(i)=\epsilon}$. Then $G$ is a triangle-free
twin-free graph and for every $(w,w')\not\in E(G)$, $N_G(w)\cap N_G(w')$
either contains an $u$-vertex or contains at least $2^{h-2}$ $v$-vertices.
Hence if we set up the weights as $\mu(u_{i\epsilon}) =\frac 1{4h}$ and
$\mu(v_a)=2^{-h-1}$, we will have $a(G,\mu)\geq \frac 1{4h}$ and $n(G)$ is
inverse exponential in $a(G,\mu)^{-1}$.
\end{remark}

Before deriving consequences mentioned in the introduction, we need a simple
exercise in linear algebra (and optimization).

\begin{lemma} \label{lem:linear}
Let $G$ be a finite graph. Then there exists at most one value $\rho=\rho_G$ for which there exist vertex weights $\mu$ such that $(G,\mu)$ is $\rho$-regular. Whenever $\rho_G$ exists, it is a rational number. Moreover, in that case there are {\bf rational} weights $\eta$ such that $(G,\eta)$ is $\rho_G$-regular and
$$
a(G,\eta) = \max \set{a(G,\mu)}{(G,\mu)\ \text{is}\ \rho_G-\text{regular}}.
$$
\end{lemma}
\begin{proof}
Fix an arbitrary system of weights $\mu$ for which $(G,\mu)$ is
$\rho$-regular for some $\rho$. Let $A$ be the adjacency matrix of $G$,
$\boldsymbol\mu$ be the (column) vector comprised of vertex weights and
$\mathbf j$ be the identically one vector. Then the regularity condition
reads as $A\boldsymbol{\mu}=\rho\cdot \mathbf j$. Since $\mathbf j$ is in the
space spanned by the columns of $A$, there exists a {\bf rational} vector
$\boldsymbol{\eta}$ such that $A\boldsymbol{\eta}=\mathbf j$. Now, on the one
hand $\boldsymbol{\eta}^T A\boldsymbol{\mu}=\rho\cdot (\boldsymbol{\eta}^T
\mathbf{j})$ and, on the other hand, $\boldsymbol{\eta}^T
A\boldsymbol{\mu}=\mathbf{j}^T \boldsymbol{\mu}=1$ (the latter equality holds
since $\mu$ is a probability measure). Hence $\rho = (\boldsymbol{\eta}^T
\mathbf{j})^{-1}$ is a rational number not depending on $\mu$.

For the second part, we note that the linear program

$$
\begin{cases}
a\to\max & ~\\
\eta(v)\geq 0 & (v\in V(G))\\
\sum_{v}\eta(v)=1 &\\
e(v)=\rho & (v\in V(G))\\
P_3^N(v,w) \geq a & ((v,w)\not\in E(G))
\end{cases}
$$
with rational coefficients in the variables $\eta(v)$ is feasible since $\mu$
is its solution. Hence it also has an optimal solution with rational
coefficients.
\end{proof}

Let us now derive consequences.

\begin{corollary} \label{cor:finiteness}
For every rational $\rho$ there exists a finite triangle-free $\rho$-regular graph $G$ such that $a(G)$ attains the maximum value $a(\rho)$ among all such graphs.
\end{corollary}
\begin{proof}
We can assume w.l.o.g. that $a(\rho)>0$. Let $\{G_n\}$ be an increasing
sequence of graphs such that $\lim_{n\to\infty} a(G_n)=a(\rho)$. Then Theorem
\ref{thm:boundedeness} implies that $\{G_n^{\text{red}}\}$ may assume only
finitely many values. Hence (by going to a subsequence) we can also assume
that all $G_n$ correspond to different vertex weights $\mu_n$ of the same
(twin-free) graph $G$. But now Lemma \ref{lem:linear} implies the existence
of rational weights $\eta(v)$, say $\eta(v)=\frac{N_v}{N}$ for integers
$N_v,N$ such that $a(G,\eta) = a(\rho)$. We convert $(G,\eta)$ to an ordinary
graph replacing every vertex $v$ with a cloud of $N_v$ twin clones.
\end{proof}

\begin{corollary} \label{cor:finiteness2}
For every $\epsilon>0$ there are only finitely many $\rho$ with
$a(\rho)\geq\epsilon$. In other words, 0 is the only accumulation point of
$\im(a)$.
\end{corollary}
\begin{proof}
Immediately follows from Theorem \ref{thm:boundedeness} and Lemma
\ref{lem:linear} since according to the latter, the edge density $\rho$ is
completely determined by the skeleton $G$ of a $\rho$-regular weighted graph
$(G,\mu)$.
\end{proof}

Now we prove that there are no ``inherently infinite'' triangle-free graphons
$W$ with $a(W)>0$. Since this result is somewhat tangential to the rest of
the paper, we will be rather sketchy and in particular we refer the reader to
\cite{Lov4} for all missing definitions.

A graphon $W\function{[0,1]\times [0,1]}{[0,1]}$ is {\em triangle-free} if
$$
\int\int\int W(x,y)W(y,z)W(x,z)dxdydz=0.
$$
Given a graphon $W$, let $P_3^N\function{[0,1]\times [0,1]}{[0,1]}$ be
defined by $P_3^N(x,y) = \int W(x,y)W(x,z)dz$; Fubini's theorem implies that
$P_3^N$ is defined a.e. and is measurable. We define $a(W)$ as the maximum
value $a$ such that
\begin{equation} \label{eq:a_w}
\lambda\of{\set{(x,y)\in [0,1]^2}{W(x,y)<1\Longrightarrow P_3^N(x,y)\geq a}}=1.
\end{equation}

To every finite vertex-weighted graph $(G,\mu)$ we can associate the
naturally defined {\em step-function graphon} $W_{G,\mu}$ (see \cite[\S
7.1]{Lov4} or Section \ref{sec:intr} above), and two graphons are {\em
isomorphic} if they have the same sampling statistics \cite[\S 7.3]{Lov4}.

\begin{theorem} \label{thm:no_graphon}
Let $W$ be  a triangle-free graphon. Then we have the following dichotomy:
either $a(W)=0$ or $W$ is isomorphic to $W_{G,\mu}$ for some finite
vertex-weighted triangle-free graph $(G,\mu)$.
\end{theorem}

\begin{proof} (sketch)
Assume that $a(W)>0$, that is \eqref{eq:a_w} holds for some $a>0$. Let
$\rn{G_n}$ be the random sample from the graphon $W$; this is a probability
measure on the set $\mathcal G_n$ of triangle-free graphs on $n$ vertices up
to isomorphism. A standard application of Chernoff's bound along with
\eqref{eq:a_w} gives us that
\begin{equation}\label{eq:bound_a}
  \prob{a(\rn{G_n})\leq a/2} \leq \exp(-\Omega(n)).
\end{equation}

Now, if we equip $\prod_{n\in\mathbb N}\mathcal G_n$ with the product measure
$\prod_n \rn{G_n}$, then the fundamental fact from the theory of graph limits
is that the sequence of graphs $\rn{G_n}$ sampled according to this measure
converges to $W$ with probability 1, and the same holds for their twin-free
reductions $\rn{G_n^{\text{red}}}$. Since the series $\sum_n\exp(-\Omega(n))$
converges, Theorem \ref{thm:boundedeness} along with \eqref{eq:bound_a}
implies that the number of vertices in $\rn{G_n^{\text{red}}}$ is bounded,
also with probability 1. Then a simple compactness argument shows that it
contains a sub-sequence converging to $W_{G,\mu}$ for some finite weighted
graph $(G,\mu)$.
\end{proof}

\section{The proof of Theorem \ref{thm:main}}

We fix a triangle-free $\rho$-regular graph $G$, and for the reasons
explained in Remark \ref{rem:large_density}, we assume that $\rho\leq\frac
13$. We have to prove that $a(G)\leq a_0(\rho)$, that is there exists a pair
of non-adjacent vertices $u,v$ with $P_3^N(u,v)\leq a_0(\rho)$. We work in
the set-up of Section \ref{sec:boundedeness}, that is we replace $G$ with its
weighted twin-free reduction $(G,\mu)$; the weights $\mu$ will be dropped
from notation whenever it may not create confusion. We also let $a\df
a(G,\mu)>0$ throughout.

\subsection{$\rho\geq \rho_1$: exploiting combinatorial structure}
\label{sec:combinatorial}

The only way in which we will be using twin-freeness is the following claim
(that was already implicitly used in the proof of Theorem
\ref{thm:boundedeness}).

\begin{claim} \label{clm:upper_bound}
For any two non-adjacent vertices $u\neq v$, $P_3^N(u,v)\leq \rho-a$.
\end{claim}
\begin{proof}
First we have $P_3^N(u,v)+\bar P_3^{N,c}(u,v) = e(v)=\rho$. Thus it remains
to prove that $\bar P_3^{N,c}(u,v)\geq a$. But since $u$ and $v$ are not
twins and $e(u)=e(v)$, there exists a vertex $w\in N_G(u)\setminus N_G(v)$.
Then $a\leq P_3^N(v,w)\leq \bar P_3^{N,c}(u,v)$, the last inequality
holds since $G$ is triangle-free.
\end{proof}

We now fix, for the rest of the proof, two non-adjacent vertices $v_1,v_2$
with $P_3^N(v_1,v_2)=a$. Let $P\df N_G(v_1)\cap N_G(v_2)$ (thus $\mu(P) =
P_3^N(u,v)=a$) and we also let $I\df V(G)\setminus (N_G(v_1)\cup N_G(v_2))$
(note that $v_1,v_2\in I$). We can easily compute $\mu(I)=I_3^N(v_1,v_2)$ by
inclusion-exclusion as follows:
\begin{equation}\label{eq:i3N}
  I_3^N(v_1,v_2) =1- e(v_1)-e(v_2)+P_3^N(v_1,v_2) =1-2\rho+a.
\end{equation}

\begin{claim} \label{clm:v3}
For any $w\in P$ there exists $v_3\in I$ such that $(w,v_3)\not\in E$.
\end{claim}
\begin{proof}
The assumptions $\rho\leq \frac 13$ and $a>0$ imply, along with
\eqref{eq:i3N}, that $I_3^N(v_1,v_2)>\rho$. As $e(w)=\rho$, Claim
\ref{clm:v3} follows.
\end{proof}

Before proceeding further, let us remark that $a_0(\rho)\geq\frac{\rho}3$ for
$\rho\in [\rho_1,1/3]$ (verifications of computationally unpleasant
statements like this one can be found in the Maple worksheet at\\ {\tt
http://people.cs.uchicago.edu/\~{}razborov/files/tfsr.mw}). Hence we can and
will assume w.l.o.g. that
\begin{equation}\label{eq:first_bound}
  a>\frac{\rho}3.
\end{equation}

\begin{claim} \label{clm:w}
For any $v_3\in I$ we have $S_4^I(v_1,v_2,v_3)>0$, that is there exists a
vertex $w\in P$ adjacent to $v_3$.
\end{claim}
\begin{proof}
Since $P$ is non-empty, we can assume w.l.o.g. that $\exists w\in P\
((v_3,w)\not \in E)$ (otherwise we are done). Now we have the computation
(again, since $G$ is triangle-free)
\begin{equation} \label{eq:S_bound}
\longeq{&&\rho = e(v_3) \geq P_3^N(v_3,v_1)+P_3^N(v_3,v_2)+ P_3^N(v_3,w)-S_4^I(v_1,v_2,v_3)
\\  && \longeqskip\geq
3a-S_4^I(v_1,v_2,v_3).}
\end{equation}
The claim now follows from \eqref{eq:first_bound}.
\end{proof}

Let now $c\df |P|$ be the {\bf size} of $P$ (weights are ignored). Claims
\ref{clm:v3} and \ref{clm:w} together imply that $c\geq 2$. The rest of the
analysis depends on whether $c=2$, $c=3$ or $c\geq 4$.

\subsubsection{$c=2$}

Let $P=\{w,w'\}$, where $\mu(w)\geq \mu(w')$, and note that $\mu(w')\leq
\frac a2$. By Claim \ref{clm:v3}, there exists $v_3\in I$ such that
$(w,v_3)\not\in E$. We have $S_4^I(v_1,v_2,v_3)\leq \mu(w')\leq \frac a2$.
Along with \eqref{eq:S_bound}, this gives us the bound
\begin{equation} \label{eq:twofifth}
a\leq \frac 25\rho.
\end{equation}

By Claim \ref{clm:w}, for any $v_3\in I$ we have either $(w,v_3)\in E(G)$ or
$(w',v_3)\in E(G)$. In other words, the neighbourhoods of $v_1,v_2,w,w'$
cover the whole graph or, equivalently, $I_3^N(v_1,v_2)+I_3^N(w,w')=1$. Now,
$I_3^N(v_1,v_2)=1-2\rho+a$ by \eqref{eq:i3N}, and for $(w,w')$ this
calculation still works in the ``right'' direction: $I_3^N(w,w')=1-2\rho
+P_3^N(w,w')\geq 1-2\rho+a$. Thus we get $a\leq 2\rho-\frac 12$. Along with
\eqref{eq:twofifth}, we get that $a\leq\min\of{\frac 25\rho, 2\rho-\frac
12}\leq a_0(\rho)$ (see the Maple worksheet) and this completes the analysis
of the case $c=2$.

\subsubsection{$c=3$} \label{sec:c3}

Let $P=\{w_1,w_2,w_3\}$. We abbreviate $F^{\mathcal I}_{\{i\}}(w_1,w_2,w_3)$
to $F_i$, $F^{\mathcal I}_{\{i,j\}}(w_1,w_2,w_3)$ to $F_{ij}$ and
$F^{\mathcal I}_{\{1,2,3\}}(w_1,w_2,w_3)$ (= $S_4^{\mathcal I}(w_1,w_2,w_3)$)
to $f_3$. In our claims below we will always assume that
$\{i,j,k\}=\{1,2,3\}$ is an arbitrary permutation on three elements.

We begin with noticing that Claim \ref{clm:upper_bound} applied to the pair
$(w_i,w_k)$ gives us $F_{ik}+f_3\leq \rho-a$ that can be re-written (since
$F_i+F_{ij}+F_{ik}+f_3=e(w_i)=\rho$) as
\begin{equation} \label{eq:ivsij}
F_i+F_{ij} \geq a.
\end{equation}
On the other hand, the bound $P_3^N(w_i,w_j)\geq a$ re-writes as
\begin{equation}\label{eq:opposite}
  F_{ij} + f_3\geq a.
\end{equation}

We also note that \eqref{eq:ivsij} (along with its analogue obtained by
changing $F_i$ to $F_j$) implies
\begin{equation} \label{eq:ivsj}
F_{ij} = 0 \Longrightarrow (F_i\geq a \land F_j\geq a).
\end{equation}

\begin{claim} \label{clm:ijvsk}
$F_{ij}>0\Longrightarrow F_k\geq a.$
\end{claim}
\begin{proof}
Let $v$ be any vertex contributing to $F_{ij}$, that is $(w_i,v), (w_j,v)\in
E(G)$ while $(w_k,v)\not\in E(G)$. Then $a\leq P_3^N(w_k,v)\leq F_k$.
\end{proof}

Now, \eqref{eq:ivsj} along with Claim \ref{clm:ijvsk} imply that there exist
at least two indices $i\in [3]$ with $F_i\geq a$. Assume w.l.o.g. that
$F_1,F_2\geq a$. Our goal (that, somewhat surprisingly, is the most
complicated part of the analysis) is to show that in fact $F_3\geq a$ as
well.

\begin{claim} \label{clm:non_zero}
$F_i>0$.
\end{claim}
\begin{proof}
When $i=1.2$, we already have the stronger fact $F_i\geq a$ so we are only
left to show that $F_3>0$. Assume the contrary. Then $F_{12}=0$ by Claim
\ref{clm:ijvsk}, hence $f_3\geq a$ by \eqref{eq:opposite}. Also, $F_{13}\geq
a$ and $F_{23}\geq a$ by \eqref{eq:ivsij} (with $i=3$). Summing all this up,
$\rho=e(w_3)=F_{13}+F_{23}+f_3\geq 3a$, contrary to the assumption
\eqref{eq:first_bound}.
\end{proof}

The next claim, as well as Claim \ref{clm:C4} below, could have been also
written very concisely at the expense of introducing a few more flags; we did
not do this since those flags are not used anywhere else in the paper.

\begin{claim} \label{clm:p4}
There is an edge between {\rm [}the sets of vertices corresponding to{\rm ]}
$F_i$ and $F_j$.
\end{claim}
\begin{proof}
Since $\{i,j\}\cap \{1,2\} \neq \emptyset$, we can assume w.l.o.g. that
$i=1$. We have
$$
\rho=e(w_1) =F_1+F_{1j}+F_{1k}+f_3
$$
and $F_1\geq a,\ F_{1j}+f_3\geq a$ (by \eqref{eq:opposite}). Hence $F_{1k}<a$ due to
\eqref{eq:first_bound}. Let now $v$ be an arbitrary vertex contributing to
$F_j$ that exists by Claim \ref{clm:non_zero}. We have $P_3^N(v,w_1)\geq a$,
and all contributions to it come from either $F_{1k}$ or $F_1$. Since
$F_{1k}<a$, $v$ must have at least one neighbor in $F_1$.
\end{proof}

\begin{claim} \label{clm:almost_last}
$F_i+F_{ij}+F_{ik}\geq 2a$.
\end{claim}
\begin{proof}
Let $v,v'$ be as in Claim \ref{clm:p4} with $i:=k$, i.e. $(v,v')\in E(G)$,
$v$ contributes to $F_k$ and $v'$ contributes to $F_j$. Then $2a\leq
P_3^N(w_i,v) + P_3^N(w_i,v')\leq F_i + F_{ij} + F_{ik}$ simply because
$(v,v')$ is an edge, and this implies that the sets corresponding to
$P_3^N(w_i,v),\ P_3^N(w_i,v')$ are disjoint.
\end{proof}

\begin{claim} \label{clm:fij}
$F_{ij}>0$.
\end{claim}
\begin{proof}
Assuming the contrary, we get $f_3\geq a$ from \eqref{eq:opposite} and
$F_i+F_{ik}\geq 2a$ from Claim \ref{clm:almost_last}. This (again)
contradicts $e(w_i)=\rho<3a$.
\end{proof}

Now we finally have
\begin{claim} \label{clm:Fi}
$F_i\geq a$.
\end{claim}
\begin{proof}
Immediate from Claims \ref{clm:ijvsk} and \ref{clm:fij}.
\end{proof}

\begin{claim} \label{clm:sophisticated}
$\mu(w_i)+F_{jk}\geq 4a-\rho$.
\end{claim}
\begin{proof}
Let (by Claim \ref{clm:Fi}) $v$ be any vertex contributing to $F_i$. Then we
have the computation (cf. \eqref{eq:S_bound}):
\begin{equation}\label{eq:exact_formula}
  \longeq{&&\rho= e(v) = T_4^{\mathcal I}(v_1,v_2,v) + P_3^N(v_1,v) + P_3^N(v_2,v) -
  S_4^{\mathcal I}(v_1,v_2,v)\\ && \longeqskip  \geq T_4^{\mathcal I}(v_1,v_2,v)+2a-\mu(w_i).}
\end{equation}
On the other hand,
\begin{equation}\label{eq:local}
  2a\leq P_3^N(v,w_j) +P_3^N(v,w_k) \leq T_4^{\mathcal I}(v_1,v_2,v) + F_{jk}
\end{equation}
(note that $v$ may not be connected to vertices in $F_{ij},F_{ik},f_3$ as it
would have created a triangle with $w_i$). The claim follows from comparing
these two inequalities.
\end{proof}

Let us now extend the notation $f_3=F_{\{1,2,3\}}$ to
$$
f_\nu \df \sum_{S \in {[3] \choose \nu}} F_S.
$$
Then Claim \ref{clm:w} implies $f_0=0$ and hence
\begin{equation} \label{eq:volume}
f_1+f_2+f_3=\mu(I) =1-2\rho+a
\end{equation}
and also
\begin{equation} \label{eq:density}
f_1+2f_2+3f_3 =\sum_i e(w_i) = 3\rho.
\end{equation}

Next, Claim \ref{clm:Fi} implies
\begin{equation} \label{eq:f1}
f_1\geq 3a
\end{equation}
and Claim \ref{clm:sophisticated}, after summing it over $i\in [3]$ gives us
\begin{equation} \label{eq:f2}
f_2\geq 11a-3\rho.
\end{equation}
Resolving \eqref{eq:volume} and \eqref{eq:density} in $f_3$, we get
\begin{equation} \label{eq:resolving}
2f_1+f_2 =3-9\rho+3a.
\end{equation}
Comparing this with \eqref{eq:f1} and \eqref{eq:f2} gives us the bound
\begin{equation} \label{eq:3_14}
a\leq \frac 3{14}(1-2\rho)
\end{equation}
which is $\leq a_0(\rho)$ as long as $\rho\in [9/32, 1/3]$.

To complete the analysis of case $c=3$ we still have to prove that
$a(\rho)\leq \mathsf{Improved}(\rho)$ for $\rho_1\leq\rho\leq \frac 9{32}$.
As it uses some material from the proof of the Krein bound, we defer this to
Section \ref{sec:improved}.

\subsubsection{$c\geq 4$}

Fix arbitrarily distinct $w_1,w_2,w_3,w_4\in P$ and let us employ the same
notation $F_i, F_{ij}, F_{ijk}$ as in the previous section;
$\{i,j,k,\ell\}=\{1,2,3,4\}$. As before, let
$$
f_\nu=\sum_{S\in {[4]\choose\nu}} F_{S}.
$$
Note that since we allow $c>4$, this time $f_0$ need not necessarily be zero. We further let
$$
\widehat F_S \df \sum_{T\subseteq [4] \atop T\cap S\neq\emptyset} F_T
$$
be the measure of $\bigcup_{i\in S} N_G(w_i)$, and we also use abbreviations
$\widehat F_i, \widehat F_{ij}, \widehat F_{ijk}, \widehat F_{1234}$ in this
case.

\smallskip
To start with, $\widehat F_i=\rho$ and Claim \ref{clm:upper_bound} implies $\widehat F_{ij}\geq \rho+a$.

\begin{claim} \label{clm:ijk}
$\widehat F_{ijk} \geq \rho +2a.$
\end{claim}
\begin{proof}
For $S\subseteq \{i,j,k\}$, let $F_S^\ast\df F_S+F_{S\cup \{\ell\}}$ be the result of ignoring $w_\ell$ and we (naturally) let
$$
f_\nu^\ast\df \sum_{S\in  {\{i,j,k\} \choose \nu}} F_S^\ast.
$$
Then (cf. \eqref{eq:density})
$$
f_1^\ast+2f_2^\ast+3f_3^\ast =3\rho,
$$
and also
$$
f_2^\ast+3f_3^\ast = P_3^N(w_i,w_j) + P_3^N(w_i,w_k) +P_3^N(w_j,w_k) \leq 3(\rho-a)
$$
by Claim \ref{clm:upper_bound}. Besides, $\widehat F_{ijk} =f_1^\ast+f_2^\ast +f_3^\ast$.

If $f_2^\ast=0$, we are done: $\widehat F_{ijk}=3\rho-2f_3^\ast\geq
3\rho-2(\rho-a)=\rho+2a$. Hence we can assume that $f_2^\ast>0$, say,
$F_{ij}^\ast>0$. Pick an arbitrary vertex $v$ corresponding to $F_{ij}^\ast$
then, as before, $\widehat F_{ijk} = \widehat F_{ij} + F_k^\ast\geq
\rho+a+P_3^N(v,w_k)\geq \rho+2a$.
\end{proof}

\begin{lemma} \label{lem:1234}
$\widehat F_{1234} \geq \rho+3a.$
\end{lemma}
\begin{proof}
First, $\widehat F_{1234}=\widehat F_{jk\ell}+F_i\geq \rho+2a+F_i$ by Claim \ref{clm:ijk}.
Hence we can assume that $F_i<a$ (for all $i\in [4]$, as usual). Also, we can assume that $f_3=0$ since otherwise we are done by the same reasoning as in the proof of Claim \ref{clm:ijk}.

Now, let $\Gamma$ be the graph on $[4]$ with the set of edges
$$
E(\Gamma) =\set{(i,j)}{F_{ij}>0}.
$$
Analogously to \eqref{eq:ivsij}, we have
\begin{equation} \label{eq:triple}
F_i+F_{ij}+F_{i\ell} \geq a
\end{equation}
(recall that $F_{ij\ell}=0$) and, analogously to Claim \ref{clm:ijvsk},
\begin{equation} \label{eq:ijvskl}
F_{ij}>0 \Longrightarrow F_k+F_{k\ell} \geq a.
\end{equation}

Next, \eqref{eq:triple}, along with $F_i<a$, implies that the minimum degree
of $\Gamma$ is $\geq 2$, that is $\Gamma$ is the complement of a matching.
Hence there are only three possibilities: $\Gamma=K_4$, $\Gamma=C_4$ or
$\Gamma= K_4-e$, and the last one is ruled out by \eqref{eq:ijvskl} along
with $F_k<a$.

If $\Gamma=K_4$ then summing up \eqref{eq:ijvskl} over all choices of
$k,\ell$, we get $3f_1+2f_2\geq 12a$. Adding this with $f_1+2f_2+4f_4=4\rho$,
we get $\widehat F_{1234} = f_1+f_2+f_4\geq \rho+3a$. Thus it remains to deal
with the case $\Gamma=C_4$, say $E(\Gamma)=\{(1,2), (2,3), (3,4), (4,1)\}$.

\smallskip
First we observe (recall that $f_3=0$) that
$$
f_4=P_3^N(w_1,w_3) (= P_3^N(w_2,w_4)) \geq a.
$$
Next, \eqref{eq:triple} amounts to
\begin{equation} \label{eq:ipm1}
F_i +F_{i,i+1} \geq a
\end{equation}
(all summations in indices are mod 4) and hence
$2F_i+F_{i,i+1}+F_{i,i-1}+f_4\geq 3a$. Comparing with
$$
F_i +F_{i,i+1} +F_{i,i-1} +f_4 = e(w_i) = \rho,
$$
we see that $F_i\geq 3a-\rho$ which is strictly positive by the assumption
\eqref{eq:first_bound}. Likewise, $F_{i,i+1}=\rho-f_4-(F_i+F_{i,i-1})\leq
\rho-2a<a$.

\begin{claim} \label{clm:C4}
There is an edge between $F_i$ and $F_{i+1}$.
\end{claim}
\begin{proofof}{Claim \ref{clm:C4}} This is similar to the proof of Claim
\ref{clm:p4}. Pick up a vertex $v$ contributing to $F_i$ ($F_i>0$ as we just
observed). Then $P_3^N(w_{i+1},v)\leq F_{i+1}+F_{i+1,i+2}$ and since we
already know that $F_{i+1,i+2}<a$, there exists a vertex corresponding to
$F_{i+1}$ and adjacent to $v$.
\end{proofof}

\begin{claim} \label{clm:last}
$F_i+F_{i+1}+F_{i,i+1}\geq 2a$.
\end{claim}
\begin{proofof}{Claim \ref{clm:last}} This is similar to the proof of Claim
\ref{clm:almost_last}. Pick vertices $v,v'$ witnessing Claim \ref{clm:C4}
with $i:=i+2$, so that in particular $(v,w_{i+2}), (v',w_{i-1}), (v,v')$ are
all in $E(G)$ while $(v,w_{i+1}), (v',w_i)$ are not. Then
$$
2a\leq P_3^N(v,w_{i+1}) + P_3^N(v',w_i) \leq F_i + F_{i+1} +F_{i,i+1}
$$
since $P_3^N(v,w_{i+1})\leq F_{i+1}+F_{i,i+1}$, $P_3^N(v',w_i)\leq F_i +
F_{i,i+1}$ and the corresponding sets are disjoint since $(v,v')$ is an edge.
\end{proofof}

Now we can complete the proof of Lemma \ref{lem:1234}:
$$
\widehat F_{1234} = (F_1 + F_{12} + F_{14} + F_{1234}) + (F_2+F_{23}) +
(F_3+F_4+F_{34}) \geq \rho+3a
$$
by \eqref{eq:ipm1} and Claim \ref{clm:last}.
\end{proof}

This also completes the proof of Theorem \ref{thm:main} for $\rho\geq \rho_1$
(that is, modulo the bound $\mathsf{Improved}(\rho)$ deferred to Section
\ref{sec:improved}). Indeed, since $\widehat F_{1234}\leq 1-2\rho+a$, Lemma
\ref{lem:1234} implies $a\leq \frac{1-3\rho}{2}$ which is $\leq a_0(\rho)$ as
long as $\rho\in [\rho_1, 1/3]$.

\subsection{Analytical lower bounds} \label{sec:analytical}

In this section we prove the bounds $a(\rho)\leq \mathsf{Krein}(\rho)\
(\rho\leq \rho_0)$, $a(\rho)\leq \widehat{\mathsf{Krein}}(\rho)\ (\rho\in
[\rho_0,\rho_1])$ and $a(\rho)\leq \mathsf{Improved}(\rho)\ (\rho\in [\rho_2,
9/32])$. We keep all the notation and conventions from the previous section.

\medskip
Let us continue a bit our crash course on flag algebras we began in Section
\ref{sec:prel}. The product $F_1(v_1,v_2,\ldots,v_k)F_2(v_1,v_2,\ldots,v_k)$,
where $F_1$ and $F_2$ are flags of the same type and $v_1,\ldots,v_k\in V(G)$
induce this type in $G$, can be always expressed as a {\em fixed} (that is,
not depending on $G,v_1,\ldots,v_k$) linear combination of expressions of the
form $F(v_1,\ldots,v_k)$. The general formula is simple (see \cite[eq.
(5)]{flag}) but it will be relatively clear how to do it in all concrete
cases we will be dealing with. We stress again that it is only possible
because we sample vertices with repetitions, otherwise the whole theory
completely breaks down. Also, things can be easily set up in such a way that,
after extending it by linearity to expressions $f(v_1,\ldots,v_k)$, where $f$
is a formal $\mathbb R$-linear combination of flags, this becomes the product
in a naturally defined commutative associative algebra.

We also need the {\em averaging} or {\em unlabelling} operator\footnote{For
the reader familiar with graph limits, let us remark that their operator is
different but connected to ours via a simple M\"obius transformation,
followed by summation over several types.} $f\mapsto \eval f{\sigma,\eta}$.
Let $\sigma$ be a type of size $k$, and $\eta\injection{[k']}{[k]}$ be an
injective mapping, usually written as $[\eta_1,\ldots,\eta_{k'}]$ or even
$\eta_1$ when $k=1$ (here $\eta_1,\ldots,\eta_{k'}$ are pairwise different
elements of $[k]$). Then we have the naturally defined type $\sigma|_\eta$ of
size $k'$ given by $(i,j)\in E\of{\sigma|_\eta}$ if and only if
$(\eta_i,\eta_j)\in E(\sigma)$. Now, given a linear combination $f$ of
$\sigma$-flags and $w_1,\ldots,w_{k'}\in V(G)$ spanning the type
$\sigma|_\eta$, we consider the expectation $\expect{f(\bar v_1,\ldots\bar
v_k)}$, where $\bar v_j$ is $w_i$ if $j=\eta_i$ and picked according to the
measure $\mu$, independently of each other, when $j\not\in\im(\eta)$.  Again,
there is a very simple general formula computing this expectation as a real
linear combination of $\sigma|_\eta$-flags, denoted by $\eval f{\sigma,
[\eta_1,\ldots,\eta_{k'}]}$ that, again, does not depend on
$G,w_1,\ldots,w_{k'}$ \cite[\S 2.2]{flag}.

\begin{remark} \label{rem:normalization}
It is important (and turns out very handy in concrete computations) to note
that we set $f(\bar v_1,\ldots, \bar v_k)\df 0$ if $\bar v_1,\ldots,\bar v_k$
do not induce $\sigma$. In particular, we let
\begin{equation}\label{eq:typeflag}
  \langle \sigma, \eta \rangle \df \eval{1}{\sigma,\eta};
\end{equation}
this is simply the pair $(\sigma,\eta)$ viewed as a $\sigma|_\eta$-flag with
an appropriate coefficient \cite[Theorem 2.5(b)]{flag}. In other words,
$\eval{f}{\sigma,\eta}$ is {\bf not} the conditional
expectation by the event ``$(\bar v_1,\ldots,\bar v_k)$ induce $\sigma$'' but
the expectation of $f$ multiplied by the characteristic function of this
event.
\end{remark}

Finally, we also need the lifting operator $\pi^{\sigma,\eta}$, where
$\sigma,\eta$ are as above. Namely, for a $\sigma|_\eta$-flag $F$, let
$$
\pi^{\sigma,\eta}(F)(v_1,\ldots,v_k) \df F(v_{\eta_1},\ldots,v_{\eta_{k'}})
$$
be the result of forgetting certain variables among $v_1,\ldots,v_k$ and
possibly re-enumerating the remaining ones according to $\eta$. It may look
trivial but we will see below that it turns out to be very handy in certain
calculations. Also note that, unlike $\eval{\cdot}{\sigma,\eta}$,
$\pi^{\sigma,\eta}$ does respect the multiplicative structure.

When $\eta$ is empty, $\eval f{\sigma,\eta}$ and $\pi^{\sigma, \eta}$ are abbreviated to $\eval f{\sigma}$ and $\pi^\sigma$, respectively.

The main tool in flag algebras is the light version of the Cauchy-Schwartz
inequality formalized as
\begin{equation} \label{eq:es}
\eval{f^2}{\sigma,\eta} \geq 0,
\end{equation}
and the power of the method relies on the fact that positive linear combinations of these inequalities can be arranged as a semi-definite programming problem. But the resulting proofs are often very non-instructive, so in this paper we have decided to use more human-oriented language of optimization. Let us stress that, if desired, the argument can be also re-cast as a purely symbolic sum-of-squares computation based on statements of the form \eqref{eq:es}.

\bigskip
After this preliminary work, let us return to the problem at hand. As in the
previous section, we fix arbitrarily two non-adjoint vertices $v_1,v_2$ with
$P_3^N(v_1,v_2)=a$ and let $P\df N_G(v_1)\cap N_G(v_2)$, $I\df V(G)\setminus
(N_G(v_1) \cup N_G(v_2))$. Recall that $\mu(P) =a$ and $\mu(I) = 1-2\rho+a$.

\subsubsection{Krein bounds} \label{sec:krein}

We are going to estimate the quantity $\eval{S_4^{\mathcal I}(T_4^{\mathcal
I} + S_4^{\mathcal I})}{\mathcal I, [1,2]}(v_1,v_2)$ from both sides and
compare results.

The upper bound does not depend on whether $\rho\leq \rho_0$ or not and it
consists of several typical flag-algebraic computations.

{\bf Convention.} When the parameters $(v_1,v_2,\ldots,v_k)$ in flags are
omitted, this means that the inequality in question holds for their arbitrary
choice. We specify them explicitly when the fact depends on the specific
property $P_3^N(v_1,v_2)=a$ of $v_1$ and $v_2$.

As we have already implicitly computed in the previous section,
$$
\eval{(S_4^{\mathcal I})^2}{\mathcal I, [1,2]} =\frac 13K_{32}^N =\frac 12\eval{K_{32}^{\mathcal P}}{\mathcal P,[1,2]}.
$$
Similarly,
$$
\eval{S_4^{\mathcal I}T_4^{\mathcal I}}{{\mathcal I,[1,2]}} = \frac 12\eval{U_5^{\mathcal P}}{\mathcal P,[1,2]}.
$$
Altogether we have
\begin{equation} \label{eq:sst}
\eval{S_4^{\mathcal I}(S_4^{\mathcal I}+T_4^{\mathcal I})}{\mathcal I,[1,2]} = \frac 12\eval{K_{32}^{\mathcal P}+U_5^{\mathcal P}}{\mathcal P, [1,2]}.
\end{equation}

On the other hand, we note that $P_3^{E,b} =\pi^{E,2}(e) =\rho$ and since $\frac 12P_3^{1,b}=\eval{P_3^{E,b}}{E,1}$, we also have $P_3^{1,b}=2\rho^2$. Hence
\begin{equation} \label{eq:kuvw}
2\rho^2 =\pi^{\mathcal P,3}(P_3^{1,b}) = K_{32}^{\mathcal P} + U_5^{\mathcal P} +V_5^{\mathcal P,1} +V_5^{\mathcal P,2}.
\end{equation}
Let us compute the right-hand side here. We have
$$
V_5^{\mathcal P,1}=2\eval{V_5^{\mathcal D,1}}{\mathcal D, [1,2,3]}
$$
\begin{equation} \label{eq:v1v2specific}
\langle \mathcal D, [1,2,3] \rangle(v_1,v_2) = \pi^{\mathcal P, [1,2]}(\bar P_3^{N,b})(v_1,v_2) =\rho-a
\end{equation}
(see the definition \eqref{eq:typeflag}) and
$$
V_5^{\mathcal D,1} = \pi^{\mathcal D,[3,4]}(P_3^N) \geq a.
$$
Putting these together,i
$$
V_5^{\mathcal P,1}(v_1,v_2,w) \geq 2a(\rho-a)\ (w\in P)
$$
and, by symmetry, the same holds for $V_5^{\mathcal P,2}$. Comparing with \eqref{eq:kuvw}, we find that
\begin{equation} \label{eq:kplusu}
(K_{32}^{\mathcal P} +U_5^\mathcal P)(v_1,v_2,w)\leq 2\rho^2-4a(\rho-a) =  2((\rho-a)^2+a^2).
\end{equation}
Averaging this over all $w\in P$ and taking into account \eqref{eq:sst}, we arrive at our first main estimate
\begin{equation} \label{eq:s4t4}
\eval{S_4^{\mathcal I}(S_4^{\mathcal I} + T_4^{\mathcal I})}{\mathcal I,[1,2]}(v_1,v_2) \leq a(\rho^2-2a(\rho-a)).
\end{equation}

\medskip
For the lower bound we first claim that
\begin{equation}\label{eq:t4I}
  T_4^{\mathcal I} \leq S_4^{\mathcal I} +\rho-2a.
\end{equation}
This was already established in \eqref{eq:exact_formula}, but let us re-cup the argument using the full notation:
$$
\rho=\pi^{\mathcal I,3}(e) = T_4^{\mathcal I} + \pi^{\mathcal I,[1,3]}(P_3^N) + \pi^{\mathcal I,[2,3]}(P_3^N)-S_4^{\mathcal I} \geq T_4^{\mathcal I}+2a-S_4^{\mathcal I}.
$$
Next, we need a lower bound on $T_4^N(v_1,v_2)=\eval{T_4^{\mathcal
I}}{\mathcal I,[1,2]}(v_1,v_2)$, that is on the density of those edges that
have both ends in $I$. For that we first classify all edges of $G$ according
to the number of vertices they have in $I$:
\begin{equation} \label{eq:rho_expansion}
\pi^N(\rho) = T_4^N + \of{S_4^N + \sum_{i=1}^2 V_4^{N,i}} +P_4^N.
\end{equation}
Now,
$$
S_4^N(v_1,v_2) = 2\eval{\pi^{\mathcal P,3}(e)}{\mathcal P,[1,2]}(v_1,v_2) = 2a\rho.
$$
Further we note that
\begin{equation}\label{eq:rhoqv}
  \rho(\rho-a) = \eval{\pi^{\mathcal Q_i,3}(e)}{\mathcal Q_i,[1,2]}(v_1,v_2)= \frac 12 \of{V_4^{N,i}+P_4^N}(v_1, v_2)\ (i=1,2).
\end{equation}
Summing this over $i=1,2$ and plugging our findings into \eqref{eq:rho_expansion}, we get
\begin{equation} \label{eq:t4n_prel}
\rho = T_4^N(v_1,v_2)+2a\rho +4\rho(\rho-a) -P_4^N(v_1,v_2).
\end{equation}
So, the only thing that still remains is to estimate $P_4^N(v_1,v_2)$ but
this time {\bf from below}. For that it is sufficient to compute its
contribution to the right-hand side of \eqref{eq:rhoqv} (letting, say,
$i:=1$):
$$
a(\rho-a) \leq \eval{\pi^{\mathcal Q_1,[2,3]}(P_3^N)}{\mathcal Q_1, [1,2]}=\frac 12 P_4^N(v_1,v_2).
$$
Substituting this into \eqref{eq:t4n_prel}, we arrive at our estimate on the number of edges entirely within $I$:
\begin{equation} \label{eq:t4n}
\longeq{&&\eval{T_4^{\mathcal I}}{\mathcal I}(v_1,v_2) =T_4^N(v_1,v_2)\\
&& \longeqskip \geq \rho-2a\rho-4\rho(\rho-a) +2a(\rho-a) = \rho - 2(\rho^2+(\rho-a)^2).}
\end{equation}

We are now prepared to bound $\eval{S_4^{\mathcal I}(S_4^{\mathcal I} +T_4^{\mathcal I})}{\mathcal I, [1,2]}(v_1,v_2)$ from below. As a piece of intuition, let us re-normalize $S_4^{\mathcal I}$ and $T_4^{\mathcal I}$ by the known values $\langle \mathcal I, [1,2]\rangle =1+a-2\rho$ (cf. Remark \ref{rem:normalization}) so that they become random variables in the triangle
$$
\mathbb T = \set{(S_4^{\mathcal I}, T_4^{\mathcal I})}{T_4^{\mathcal I}\geq 0,\ T_4^{\mathcal I}\leq S_4^{\mathcal I}+\rho-2a,\ S_4^{\mathcal I}\leq a}.
$$
Then we know the expectation of $S_4^{\mathcal I}$, have the lower bound
\eqref{eq:t4n} on the expectation of $T_4^{\mathcal I}$,  and we need to
bound the expectation of $S_4^{\mathcal I}(S_4^\mathcal I + T_4^{\mathcal
I})$, also from below. For that purpose we are going to employ duality, i.e.
we are looking for coefficients $\alpha,\beta,\gamma$ depending on $a,\rho$
only such that
$$
L(x,y) \df x(x+y)-(\alpha x+\beta y +\gamma)
$$
is non-negative on $\mathbb T$, and applying $\eval{\cdot}{\mathcal I,
[1,2]}$ to this relation produces ``the best possible result''. As we
mentioned above, an alternative would be to write down an explicit
``sum-of-squares'' expression: the resulting proof would be shorter but it
would be less intuitive.

Let us first observe the obvious upper bound
\begin{equation} \label{eq:trivial_bound}
a\leq \frac{\rho^2}{1-\rho},
\end{equation}
it follows from the computation $3\rho^2=3\eval{P_3^1}1 = P_3 =3
\eval{P_3^N}N\geq 3a(1-\rho)$. Next, the right-hand side of \eqref{eq:t4n} is
a concave quadratic function in $a$, with two roots $a_1(\rho)\df
\rho-\frac{\sqrt{2\rho-4\rho^2}}{2}$, $a_2(\rho)\df
\rho-\frac{\sqrt{2\rho+4\rho^2}}{2}$. Further, $a_1(\rho)\leq a_0(\rho)\leq
\frac{\rho^2}{1-\rho}\leq a_2(\rho)$. Hence we can assume w.l.o.g. that the
right-hand side in \eqref{eq:t4n} is non-negative. Therefore, by decreasing
$T_4^{\mathcal I}$ if necessary, we can assume that the bound \eqref{eq:t4n}
on its expectation is actually tight.

Next, we note that since the quadratic form $x(x+y)$ is indefinite, the
function $L(x,y)$ attains its minimum somewhere on the border of the compact
region $\mathbb T$. Since $L$ is linear on the line $x=a$ we can further
assume that the minimum is attained at one of the lines $y=0$ or
$y=x+\rho-2a$. Note further that along both these lines $L$ is convex.

\bigskip
We begin more specific calculations with the bound $g_K(\rho,a)\geq 0$ that
is less interesting but also less computationally heavy. As a motivation for
the forthcoming computations, we are looking for two points $(x_0,0)$,
$(x_1,x_1+\rho-2a)$ on the lines $T_4^{\mathcal I}=0$, $T_4^{\mathcal I} =
S_4^{\mathcal I}+\rho-2a$ that are collinear\footnote{cf. \eqref{eq:t4n}, the
normalizing factor $1-2\rho+a$ is suggested by Remark
\ref{rem:normalization}. The particular choice of $c_x,c_y$ is needed only
for the ``best possible result'' part.} with the point $(c_x,c_y)$, where
$$
c_x\df \frac{a\rho}{1-2\rho+a},\ \ \ c_y\df \frac{\rho-2(\rho^2+(\rho-a))^2}{1-2\rho+a}.
$$
and such that the function $L(x,0)$ has a double root at $x_0$ while $L(x,x+\rho-2a)$ has a double root at $x_1$. Solving all this in $\alpha,\beta,\gamma,x_0,x_1$ gives us (see the Maple worksheet)
\begin{eqnarray}
\label{eq:x0} x_0 &=& c_x+(\sqrt 2 -1)c_y\\
\nonumber x_1 &=& \of{1-\frac{\sqrt 2}2}((\sqrt 2+1)x_0-(\rho-2a))\\
\nonumber \alpha &=& 2x_0\\
\nonumber \beta &=& (3-2\sqrt 2)(2(\sqrt 2+1)x_0-(\rho-2a))\\
\nonumber \gamma &=& -x_0^2.
\end{eqnarray}
The remarks above imply that indeed $L(x,y)|_{\mathbb T}\geq 0$ hence we have
\begin{equation} \label{eq:st4_lower_2}
\eval{S_4^{\mathcal I}(S_4^{\mathcal I} + T_4^{\mathcal I})}{\mathcal I, [1,2]}
\geq \alpha a\rho +\beta (\rho-2(\rho^2+(\rho-a)^2) +\gamma(1-2\rho+a).
\end{equation}
Comparing this with \eqref{eq:s4t4}, we get (up to the positive
multiplicative factor $\frac{1-2\rho+a}{2}$) that $g_K(\rho,a)\geq 0$.
 Given the way the function $\widehat{\mathsf{Krein}}$ was
defined, $g_K(\rho,a)<0$ whenever $a\in\of{\widehat{\mathsf{Krein}}(\rho),
\frac{\rho^2}{1-\rho}}$. The required bound $a\leq
\widehat{\mathsf{Krein}}(\rho)$ now follows from \eqref{eq:trivial_bound}.

\medskip
The improvement $f_K(\rho,a)\geq 0$ takes place when the right-hand side in
\eqref{eq:x0} is $>a$ since then we can hope to utilize the condition
$S_4^\mathcal I\leq a$. As above, we first explicitly write down a solution
of the system obtained by replacing the equation $L'(x,0)|_{x=x_0}=0$ with
$x_0=a$ and only then justify the result.

Performing the first step in this program gives us somewhat cumbersome
rational functions that we attempt to simplify by introducing the
abbreviations
\begin{eqnarray*}
u_0(\rho,a) &\df& \frac 17(\rho+2a-2a\rho-4\rho^2)\\
u_1(\rho,a) &\df& \frac 17(3\rho-a-7a^2+15a\rho-12\rho^2)\\
u(\rho,a) &\df& 4u_0(\rho,a) + u_1(\rho,a).
\end{eqnarray*}
Then we get
\begin{eqnarray*}
x_0 &=& a\\
x_1 &=& \frac{a(2a-\rho^2-3\rho a)}{u(\rho,a)}\\
\alpha &=& 2a+ \frac{7(\rho-a)(u_1(\rho,a)^2-2u_0(\rho,a)^2)}{u(\rho,a)^2}\\
\beta &=& \frac{a(34u_0(\rho,a)^2+3u_1(\rho,a)^2-4u_0(\rho,a)u_1(\rho,a)-
2a\rho(1-3\rho+a)^2)}{u(\rho,a)^2}\\
\gamma &=& a^2-\alpha a.
\end{eqnarray*}

In order to analyze this solution, we first note that due to the bound just
established we can assume w.l.o.g. that
$$
a \in [\mathsf{Krein}(\rho), \widehat{\mathsf{Krein}}(\rho)].
$$

The function $u_0(\rho,a)$ is linear and increasing in $a$ and $u_0\of{\rho,
\mathsf{Krein}(\rho)}>0\ (\rho\neq 0)$ hence $u_0(\rho,a)\geq 0$. The
function $u_1(\rho,a)$ is quadratic concave in $a$ and $u_1\of{\rho,
\mathsf{Krein}(\rho)}, u_1\of{\rho, \widehat{\mathsf{Krein}}(\rho)}\geq 0$.
These two facts imply that $u(\rho,a)>0$ ($\rho>0$) hence our functions are
at least well-defined.

Next, $u_0,u_1\geq 0$ imply that $L'(x,0)|_{x=a} = 2a-\alpha$ has the sign
opposite to $u_1(\rho,a)-\sqrt 2u_0(\rho,a)$. This expression (that up to a
constant positive factor is equal to $c_x+(\sqrt 2-1)c_y-a$) is also concave
in $a$. Moreover, it is non-negative for $\rho \in [0,\rho_0],\ a \in
[\mathsf{Krein}(\rho), \widehat{\mathsf{Krein}}(\rho)]$ (at $\rho=\rho_0$ the
two bounds meet together: $\mathsf{Krein}(\rho_0) =
\widehat{\mathsf{Krein}}(\rho_0)=\rho_0/3$ and also $u_1(\rho,\rho/3)-\sqrt
2u_0(\rho,\rho/3)=0$). This completes the proof of $L'(x,0)|_{x=a}\leq 0$
hence (given that $L(a,0)=0$) we have $L(x,0)\geq 0$ for $x\leq a$. As we
argued above, this gives us $L|_{\mathbb T}\geq 0$ which implies
\eqref{eq:st4_lower_2}, with new values of $\alpha,\beta,\gamma$. Comparing
it with \eqref{eq:t4n}, we get $f_K(\rho,a)\geq 0$, up to the positive
multiplicative factor $\frac{2a(\rho-a)}{u(\rho,a)}$. This concludes the
proof of $f_K(\rho,a)\geq 0$ whenever $\rho\leq\rho_0$ and hence of the bound
$a\leq\mathsf{Krein}(\rho)$ in that interval.

\medskip
As a final remark, let us note that since the final bound $f_K(\rho,a)\geq 0$
has a very clear meaning in algebraic combinatorics, it looks likely that the
disappointingly complicated expressions we have encountered in proving it
might also have a meaningful interpretation. But we have not pursued this
systematically.

\subsubsection{The improved bound for $c=3$} \label{sec:improved}

Let us now finish the proof of the bound $a\leq \mathsf{Improved}(\rho),\
\rho \in [\rho_2,9/32]$ left over from Section \ref{sec:c3}. We utilize all
the notation introduced there, assume that $c=3$, and we need to prove that
$f_I(\rho,a)\geq 0$. We also introduce the additional notation
$$
a_i \df \mu(w_i)\ (i=1..3)
$$
for the weights of the vertices comprising the set $P$; thus, $\sum_{i=1}^3
a_i=a$.

We want to obtain an upper bound on $T_4^N(v_1,v_2)$ and then compare it with
\eqref{eq:t4n}. Let us split $I = J\stackrel .\cup K$, where $J$ corresponds
to $f_1$ and $K$ corresponds to $f_2+f_3$. Recalling that
$$
T_4^N = \eval{T_4^\mathcal I}{\mathcal I, [1,2]},
$$
let us split the right-hand side according to this partition as (with slight
abuse of notation)
$$
\eval{T_4^\mathcal I}{\mathcal I, [1,2]} = \eval{T_4^\mathcal I}{\mathcal J,
[1,2]} + \eval{T_4^\mathcal I}{\mathcal K,
[1,2]}.
$$

When $v\in J$ corresponds to $F_i$, we have $S_4^{\mathcal I}(v_1,v_2,v) =
a_i$ and hence, by \eqref{eq:t4I}, $T_4^{\mathcal I}(v_1,v_2,v)\leq
\rho-2a+a_i$. Thus
$$
\eval{T_4^\mathcal I}{\mathcal J, [1,2]} \leq \sum_i F_i(\rho-2a+a_i).
$$

In order to bound $\eval{T_4^\mathcal I}{\mathcal K, [1,2]}$, we first note
that $K$ is independent (every two vertices in $K$ have a common neighbor in
$P$). Furthermore, the only edges between $K$ and $J$ are between parts
corresponding to $F_i$ and $F_{jk}$. Hence $\eval{T_4^\mathcal I}{\mathcal K,
[1,2]}\leq \sum_i F_iF_{jk}$ and we arrive at the bound
\begin{equation} \label{eq:symmetric}
T_4^N(v_1,v_2) \leq \sum_i F_i(\rho-2a+F_{jk} +a_i).
\end{equation}

Next, let us denote by $\epsilon_i$ the (non-negative!) deficits in Claim
\ref{clm:sophisticated}:
$$
\epsilon_i \df a_i+F_{jk} - 4a+\rho;\ \epsilon_i\geq 0.
$$
Then  \eqref{eq:symmetric} re-writes as follows:
$$
  T_4^N(v_1,v_2) \leq 2af_1+\sum_{i=1}^3 F_i\epsilon_i.
$$

Let us now assume w.l.o.g. that $F_1\geq F_2\geq F_3$. Then, since all
$\epsilon_i$ are non-negative,
$$
\sum_{i=1}^3 F_i\epsilon_i \leq F_1\cdot \sum_{i=1}^3\epsilon_i = F_1
(f_2-11a+3\rho) = F_1(3-6\rho-8a-2f_1),
$$
where the last equality follows from \eqref{eq:resolving}. Summarizing,
\begin{eqnarray*}
T_4^N(v_1,v_2) &\leq& 2af_1+F_1(3-6\rho-8a-2f_1) = F_1(3-6\rho-8a) -
2f_1(F_1-a)\\ &\leq& F_1(3-6\rho-8a) - 2(F_1+2a)(F_1-a),
\end{eqnarray*}
where the last inequality holds since $F_1\geq a$ and $f_1=F_1+F_2+F_3\geq
F_1+2a$ by Claim \ref{clm:Fi}. The right-hand side here is a concave
quadratic function in $F_1$; maximizing, we find
$$
T_4^N(v_1,v_2)\leq \frac{33}2a^2+15a\rho-\frac{15}2a+\frac 92\rho^2 -\frac
92\rho +\frac 98.
$$
Comparing with \eqref{eq:t4n}, we get a constraint $Q(\rho,a)\geq 0$ that is
quadratic concave in $a$, and $\mathsf{Improved}(\rho)$ is its smallest root.
Moreover, $Q\of{\rho, \frac 3{14}(1-2\rho)} =
-\frac{(11\rho-2)(9-32\rho)}{49}\leq 0$ since $\rho_2>\frac 2{11}$. Hence the
preliminary bound \eqref{eq:3_14} can be improved to $a\leq
\mathsf{Improved}(\rho)$.

\section{Conclusion}
In this paper we have taken a prominent open problem in the algebraic graph
theory and considered its natural semi-algebraic relaxation in the vein of
extremal combinatorics. The resulting extremal problem displays a remarkably
rich structure, and we proved upper bounds for it employing methods greatly
varying depending on the range of edge density $\rho$. Many of these methods
are based on counting techniques typical for extremal combinatorics, and one
bound has a clean interpretation in terms of algebraic Krein bounds for the
triangle-free case.

\medskip
The main generic question left open by this work is perhaps how far can this
connection between the two areas go. Can algebraic combinatorics be a source
of other interesting extremal problems? In the other direction, perhaps flag
algebras and other advanced techniques from extremal combinatorics can turn
out to be useful for ruling out the existence of highly symmetric
combinatorial objects with given parameters? These questions are admittedly
open-ended so we would like to stop it here and conclude with several
concrete open problems regarding TFSR graphs and their relaxations introduced
in this paper.

Can the Krein bound $a(\rho)\leq \mathsf{Krein}(\rho)$ be improved for small
values of $\rho$? Of particular interest are the values $\rho=\frac{16}{77}$,
$\rho=\frac 5{28}$ or $\rho =\frac 7{50}$, ideally showing that
$a\of{\frac{16}{77}}= \frac 4{77}$, $a\of{\frac{5}{28}}= \frac 1{28}$ or
$a\of{\frac{7}{50}}= \frac 1{50}$. In other words, can we show that like the
four denser TFSR graphs, the M22 graph, the Gewirtz graph and the Hoffman
graph are also extremal configurations for their respective edge densities?

Another obvious case of interest is $\rho=\frac{57}{3250}$ corresponding to
the only hypothetical unknown Moore graph. More generally, can we rule out
the existence of a TSFR graph for at least one additional pair $(\rho,a)$ by
showing that actually $a(\rho)\leq a$?

For some ``non-critical'' (that is, not corresponding to TFSR graphs) $\rho$
it is sometimes also possible to come up with constructions providing
non-trivial {\bf lower} bounds on $a(\rho)$. A good example\footnote{Let us
remind that we confine ourselves to the region $\rho\leq 1/3$. A complete
description of all non-zero values $a(\rho)$ for $\rho>1/3$ follows from
\cite{BrT}.} is provided by the Kneser graphs $\text{KG}_{3k-1,k}$ having
$\rho=\frac{{2k-1 \choose k}}{{3k-1\choose k}}$ and $a=\frac{1}{{3k-1\choose
k}}$ but there does not seem to be any reasons to believe that they are
optimal. Are there any other values of $\rho$ for which we can compute
$a(\rho)$ exactly? Of particular interest here is the value $\rho=1/3$
critical for the Erd\"os-Simonovits problem (see again \cite[Problem 1]{BrT}
and the literature cited therein). Can we compute $a(1/3)$ or at least
determine whether $a(1/3)=0$ or not?

Speaking of which, is there {\bf any} rational $\rho\in (0,1/3]$ for which
$a(\rho)=0$? Equivalently, does there exist $\rho\in [0,1/3]$ for which there
are no triangle-free $\rho$-regular graphs (or, which is the same, weighted
twin-free graphs) of diameter 2? Note for comparison that there are many such
values for $\rho>1/3$; in fact, all examples leading to non-zero $a(\rho)$
fall into one of a few infinite series.

We conclude by remarking in connection with this question that regular
weighted triangle-free twin-free graphs of diameter 2 seem to be extremely
rare: a simple computer search has shown that Petersen is the {\bf only} such
graph on $\leq 11$ vertices with $\rho\leq 1/3$.

\bibliographystyle{alpha}
\bibliography{razb}
\end{document}